\documentclass[11pt]{amsart}
\usepackage{amsxtra}
\usepackage{amssymb}
\usepackage[mathscr]{eucal}
\theoremstyle{plain}
\usepackage{color}

\newtheorem{theorem}{Theorem}[subsection]
\newtheorem{lemma}[theorem]{Lemma}
\newtheorem{definition-theorem}[theorem]{Definition-Theorem}
\newtheorem{proposition}[theorem]{Proposition}
\newtheorem{corollary}[theorem]{Corollary}
\newtheorem{definition}[theorem]{Definition}
\newtheorem{example}[theorem]{Example}
\newtheorem{remark}[theorem]{Remark}
\newtheorem{conjecture}[theorem]{Conjecture}
\newtheorem{notation}[theorem]{Notation}

\newtheorem{assumption}[theorem]{Assumption}
\newtheorem*{maintheorem*}{Main Theorem}
\newcommand \bth[1] { \begin{theorem}\label{t#1} }
\newcommand \ble[1] { \begin{lemma}\label{l#1} }

\newcommand \bpr[1] { \begin{proposition}\label{p#1} }
\newcommand \bco[1] { \begin{corollary}\label{c#1} }
\newcommand \bde[1] { \begin{definition}\label{d#1}\rm }
\newcommand \bex[1] { \begin{example}\label{e#1}\rm }
\newcommand \bre[1] { \begin{remark}\label{r#1}\rm }
\newcommand \bcj[1] { \begin{conjecture}\label{j#1}\rm }

\newcommand \bnota[1] { \begin{notation}\label{n#1}\rm }
\renewcommand {\eth} { \end{theorem} }
\newcommand {\ele} { \end{lemma} }

\newcommand {\epr} { \end{proposition} }
\newcommand {\eco} { \end{corollary} }
\newcommand {\ede} { \end{definition} }
\newcommand {\eex} { \end{example} }
\newcommand {\ere} { \end{remark} }
\newcommand {\ecj} { \end{conjecture} }

\newcommand {\enota} { \end{notation} }
\newcommand \thref[1]{Theorem \ref{t#1}}
\newcommand \leref[1]{Lemma \ref{l#1}}
\newcommand \prref[1]{Proposition \ref{p#1}}

\newcommand \deref[1]{Definition \ref{d#1}}

\makeatletter
\newsavebox{\@brx}
\newcommand{\llangle}[1][]{\savebox{\@brx}{\(\m@th{#1\langle}\)}%
  \mathopen{\copy\@brx\kern-0.5\wd\@brx\usebox{\@brx}}}
\newcommand{\rrangle}[1][]{\savebox{\@brx}{\(\m@th{#1\rangle}\)}%
  \mathclose{\copy\@brx\kern-0.5\wd\@brx\usebox{\@brx}}}
\makeatother
\usepackage{tikz}
\usetikzlibrary{matrix}
\usepackage{graphicx,ctable,booktabs}
\usetikzlibrary{arrows.meta}
\usepackage{tikz-cd}

\DeclareMathOperator{\Ext}{Ext}

 \DeclareMathOperator{\Proj}{Proj}

\DeclareMathOperator{\Aut}{Aut}

 \DeclareMathOperator{\Hom}{Hom}

\DeclareMathOperator{\ThickId}{ThickId}
\DeclareMathOperator{\Mod}{{\sf Mod}}
\DeclareMathOperator{\modd}{{\sf mod}}
\DeclareMathOperator{\stmod}{{\sf stmod}}
\DeclareMathOperator{\StMod}{{ \sf StMod}}
\DeclareMathOperator{\Spc}{Spc}

\DeclareMathOperator{\ev}{ev}
\DeclareMathOperator{\coev}{coev}

\DeclareMathOperator{\For}{For}

\newcommand{\mf}{\mathfrak}
\newcommand{\mc}{\mathcal}

\newcommand{\id}{\operatorname{id}}

\newcommand{\kk}{\Bbbk}
\newcommand{\bT}{\mathbf T}
\newcommand{\bR}{\mathbf R}
\newcommand{\bS}{\mathbf S}
\newcommand{\bC}{\mathbf C}
\newcommand{\bK}{\mathbf K}

\newcommand{\bP}{\mathbf P}
\newcommand{\bQ}{\mathbf Q}
\newcommand{\bI}{\mathbf I}
\newcommand{\bJ}{\mathbf J}
\newcommand{\bD}{\mathbf D}
\newcommand{\CC}{\mathcal{C}}
\newcommand{\Loc}{\operatorname{Loc}}
\newcommand{\XX}{\mathcal X}
\newcommand{\VV}{\mathcal V}
\newcommand{\unit}{\ensuremath{\mathbf 1}}

\newcounter{listequation}
\newenvironment{eqnlist}{\begin{list}
{(\thesubsection.\thelistequation)}
{\usecounter{listequation} \setlength{\itemsep}{1.0ex plus.2ex
minus.1ex}}\setcounter{listequation}{\value{equation}}}
{\setcounter{equation}{\value{listequation}}\end{list}}

\numberwithin{equation}{subsection}

\begin{document}

\title{Noncommutative Tensor Triangular Geometry}

\author[Daniel K. Nakano]{Daniel K. Nakano}
\address{Department of Mathematics \\
University of Georgia \\
Athens, GA 30602\\
U.S.A.}
\thanks{Research of D.K.N. was supported in part by NSF grant DMS-1701768.}
\email{nakano@math.uga.edu}

\author[Kent B. Vashaw]{Kent B. Vashaw}
\address{
Department of Mathematics \\
Louisiana State University \\
Baton Rouge, LA 70803 \\
U.S.A.}
\thanks{Research of K.B.V. was supported by a Board of Regents LSU fellowship and in part by NSF grant DMS-1901830.}
\email{kvasha1@lsu.edu}
\author[Milen T. Yakimov]{Milen T. Yakimov}
\address{
Department of Mathematics \\
Northeastern University \\
Boston, MA 02115 \\
U.S.A.}
\thanks{Research of M.T.Y. was supported in part by NSF grant DMS-1901830 and DMS-2131243, and a Bulgarian Science Fund grant DN02/05.}
\email{m.yakimov@northeastern.edu}
\begin{abstract}
We develop a general noncommutative version of Balmer's tensor triangular geometry that 
is applicable to arbitrary monoidal triangulated categories (M$\Delta$Cs). Insight 
from noncommutative ring theory is used to obtain a framework for prime, semiprime, and 
completely prime (thick) ideals of an M$\Delta$C, $\bK$, and then to associate to $\bK$ a topological space--the 
Balmer spectrum $\Spc \bK$. We develop a general framework 
for (noncommutative) support data, coming in three different flavors, and show that $\Spc \bK$ 
is a universal terminal object for the first two notions (support and weak support). The first two types of support data 
are then used in a theorem that gives a method for the explicit classification of the thick (two-sided) ideals and the Balmer 
spectrum of an M$\Delta$C. The third type (quasi support) is used in another theorem that provides a method for the 
explicit classification of the thick right ideals of $\bK$, which in turn can be applied to classify the 
thick two-sided ideals and $\Spc \bK$. 

As a special case, our approach can be applied to the stable module categories of arbitrary finite
dimensional Hopf algebras that are not necessarily cocommutative (or quasitriangular). We illustrate 
the general theorems with classifications of the Balmer spectra and 
thick two-sided/right ideals for the stable module categories of all small quantum 
groups for Borel subalgebras, and classifications of the Balmer spectra and thick two-sided ideals of Hopf algebras studied by Benson and Witherspoon.
\end{abstract}
\maketitle
%%%%%%%%%%%%%%%%%
\section{Introduction}
\label{Intro}
\subsection{From the commutative to the noncommutative settings} 
Balmer's tensor triangular geometry \cite{Balmer1,Balmer2} provides a powerful method for addressing problems in representation theory, 
algebraic geometry, commutative algebra, homotopy theory and algebraic $K$-theory from one common perspective. 
In representation theory alone this covers modular representations, representations of finite group schemes, 
supergroups, quasitriangular quantum groups at roots of unity and others. 

The general setting of \cite{Balmer1,Balmer2} is that of a triangulated category $\bK$ with a biexact monoidal structure that is symmetric 
(or more generally braided). The key ingredients of Balmer's theory are:
\begin{enumerate}
\item[$\bullet$] A construction of a topological space $\Spc \bK$ consisting of prime thick ideals of $\bK$ (that is upgraded to a ringed space) 
and a characterization theorem that it is the universal final object for support maps on objects of $\bK$; 
\vskip .1cm 
\item[$\bullet$]  Methods for the explicit description of $\Spc \bK$ as a topological space via (cohomological) support data for $\bK$.  
\end{enumerate}

In this paper we develop general noncommutative versions of these ingredients of Balmer's theory that deal with an arbitrary monoidal triangulated 
category $\bK$ (M$\Delta$C for short)--a triangulated category $\bK$ with a biexact monoidal structure. Noncommutative versions of Balmer's theory 
were sought after before, because there are many important M$\Delta$Cs (e.g., the stable module categories of finite-dimensional Hopf algebras 
which are not cocommutative). However, for various reasons, a general noncommutative version of  tensor triangular geometry 
has not been fully developed. The key new ideas of our constructions are:
\begin{enumerate}
\item[$\bullet$] Previous considerations of cohomological support maps in the noncommutative setting \cite{BW1,PW} focused on the fact that they do not 
satisfy the usual axioms for commutative support data from \cite{Balmer1}. These axioms are object-wise conditions for $\bK$ and  
mimic the treatment of completely prime ideals in a noncommutative ring. Such ideals are too few in general.
The novel feature of our approach is to define the noncommutative Balmer spectrum $\Spc \bK$ and support data for $\bK$ in terms of 
tensoring of thick ideals of $\bK$, and not object-wise tensoring.
\vskip .1cm
\item[$\bullet$] In noncommutative ring theory, the prime spectrum of a ring is very hard to describe as 
a topological space with the primary example being that of the spectra of universal enveloping algebras of Lie algebras \cite{Dixmier}
and quantum groups \cite{Joseph}.
In the categorical setting, we present a method for computing the Balmer spectrum $\Spc \bK$ that appears to be as applicable as its
commutative counterparts. 
\vskip .1cm 
\item[$\bullet$] The set of right ideals of a noncommutative ring are rarely classifiable with the exceptions of very few rings. 
Surprisingly, in the categorical setting, we are successful in developing a method for the explicit classification of the thick right ideals of an M$\Delta$C.
\end{enumerate}
%%%%%%%%%%%%%%%
\subsection{Different types of prime ideals of an M$\Delta$C and its Balmer spectrum} We define below the various notions of 
primeness used throughout the paper. 
\begin{enumerate}
\item [(i)] 
A {\it {thick (two-sided) ideal}} of an M$\Delta$C, $\bK$, 
is a full triangulated subcategory of $\bK$ that contains all direct summands of its objects and is closed under left and right
tensoring with arbitrary objects of $\bK$; denote by $\ThickId(\bK)$ the set of those; 
\item[(ii)] A {\it {prime}} ideal of $\bK$ is a proper thick ideal $\bP$ such that $\bI \otimes \bJ \subseteq \bP \Rightarrow \bI \subseteq \bP$ or $\bJ \subseteq \bP$ 
for all thick ideals $\bI$ and $\bJ$ of $\bK$.  
\item[(iii)] A {\it {semiprime}} ideal of $\bK$ is an intersection of prime ideals of $\bK$.
\item[(iv)] A {\it {completely prime ideal}} of $\bK$ is a proper thick ideal $\bP$ such that $A \otimes B \in \bP \Rightarrow A \in \bP$ or $B \in \bP$ 
for all $A, B \in \bK$.   
\item[(v)] The {\it {noncommutative Balmer spectrum}} of $\bK$ is the topological space of prime ideals of $\bK$ with the Zariski topology having closed sets 
\[
V(\mc{S})=\{ \bP \in \Spc \bK | \bP \cap \mc{S}= \varnothing \}
\]
for all subsets $\mc{S}$ of $\bK$.
\end{enumerate}
\begin{theorem} \label{T:theoremA} The following hold for an M$\Delta$C, $\bK$:
\begin{enumerate}
\item[(a)] A proper thick ideal $\bP$ of $\bK$ is prime if and only if $A \otimes C \otimes B \in \bP$, for all $C \in \bK \Rightarrow A \in \bP$ or $B \in \bP$ for all $A, B \in \bK$.  
\item[(b)] A proper thick ideal $\bP$ of $\bK$ is semiprime if and only if $A \otimes C \otimes A \in \bP$, for all $C \in \bK \Rightarrow A \in \bP$ for all $A \in \bP$.
\item[(c)] The Balmer spectrum $\Spc \bK$ is always nonempty.
\end{enumerate}
\end{theorem}
The ideal-theoretic definitions of prime and semiprime ideals of an M$\Delta$C and its noncommutative Balmer spectrum were introduced by 
Buan, Krause and Solberg in \cite{BKS}. The equivalent formulations of these objects in terms of object-theoretic conditions given by Theorem \ref{T:theoremA}
will play a key role in this paper in the definitions of support data maps of various types and theorems for the effective reconstruction of the 
noncommutative Balmer spectrum of an M$\Delta$C.
%%%%%%%%%%%%%
\subsection{Support data and universality of the noncommutative Balmer spectrum} 
\label{supp}
For a topological space $X$, let $\mc{X}$, $\mc{X}_{cl}$, and $\mc{X}_{sp}$ denote the collection of all of its subsets, closed subsets and 
specialization closed subsets, respectively (see Section \ref{noncomm-supp}). The three different kinds of noncommutative support data for an M$\Delta$C, $\bK$, 
will be maps 
\[
\sigma : \bK \to \mc{X}
\]
that respect the triangulated structure:
\begin{enumerate}
\item[(i)] $\sigma(0)=\varnothing $ and $ \sigma(1)= X$;
\item[(ii)] $\sigma(A\oplus B)=\sigma(A)\cup \sigma(B)$, $\forall A, B \in \bK$; 
\item[(iii)] $\sigma(\Sigma A)=\sigma(A)$, $\forall A \in \bK$; 
\item[(iv)] If $A \to B \to C \to \Sigma A$ is a distinguished triangle, then $\sigma(A) \subseteq \sigma(B) \cup \sigma(C)$.
\end{enumerate}
\smallskip
In this paper, each kind of support datum will satisfy one additional condition related to the monoidal structure. We present an overview for 
each type below. For all of them the key role will be played by their extensions
$\Phi_\sigma$ to the sets of thick ideals $\ThickId(\bK)$ (and thick right ideals of $\bK$), given by  
\[
\Phi_\sigma(\bI) = \bigcup_{A \in \bI} \sigma(A).
\]
The strictest notion, that of {\it{support datum}}, is a map $\sigma : \bK \to \mc{X}$ satisfying the additional assumption
\smallskip
\begin{enumerate}
\item[(v)] $\bigcup_{C \in \bK} \sigma (A   \otimes C \otimes B) = \sigma (A) \cap \sigma(B)$, $\forall A, B \in \bK$. 
\end{enumerate}
\smallskip
The map $V$ defining the Zariski topology of $\Spc \bK$ is an example of a support datum, and the condition (v) is nothing but a restatement 
of the characterization of a prime ideal in Theorem~\ref{T:theoremA}(a). 

A {\it{weak support datum}} is a map $\sigma : \bK \to \mc{X}$ that  (in place of (v)) satisfies the property
\smallskip
\begin{enumerate}
\item[(v')] $\Phi_\sigma(\bI \otimes \bJ) = \Phi_\sigma(\bI) \cap \Phi_\sigma(\bJ)$ for all thick ideals $\bI$ and $\bJ$ of $\bK$.
\end{enumerate}
\smallskip
Each support datum is a weak support datum.

At various times in the paper, for each type of support datum (resp. weak support datum, and quasi support datum), we make a minor modification in 
the definition by replacing (ii) with 
\smallskip
\begin{enumerate} 
\item[(ii')] $\sigma(\bigoplus_{i\in I} A_i)=\bigcup_{i \in I} \sigma(A_i )$, $\forall A_i \in \bK$
\end{enumerate} 
\smallskip
for M$\Delta$Cs, $\bK$ admitting arbitrary set indexed coproducts.
With this replacement, we will use the term {\em{extended support datum}} (resp. {\em{extended weak support datum}}, and {\em{extended quasi support datum}}). 

\begin{theorem} \label{T:theoremB}
Let $\bK$ be an M$\Delta$C. 
\begin{enumerate}
\item[(a)] The support $V$ is the final object in the collection of support data $\sigma$ such that $\sigma(A)$ is closed for each $A \in \bK$: 
for any such $\sigma$, there is a unique continuous map $f_\sigma: X \to \Spc \bK$ satisfying 
\[
\sigma(A)= f_\sigma^{-1}( V(A)) \quad
\mbox{for} 
\quad A \in \bK.
\]
\item[(b)] The support $V$ is the final object in the collection of weak support datum $\sigma$ such that $\Phi_{\sigma}(\langle A \rangle)$ is closed for each $A \in \bK$: 
for any such $\sigma$, there is a unique continuous map $f_\sigma: X \to \Spc \bK$ satisfying 
\[
\Phi_{\sigma}(\langle A \rangle)= f_\sigma^{-1}(V(A)) \quad
\mbox{for} 
\quad A \in \bK,
\]
where $\langle A \rangle$ denotes the smallest thick ideal of $\bK$ containing $A$. 
\end{enumerate}
\end{theorem}
A version of the second part of this theorem was obtained in \cite[Theorem 4.2]{BKS}.

\subsection{Classification methods for the noncommutative Balmer spectra} 
\label{class-Balmer-spec}
In the opposite direction, the different types of support data can be used for explicit descriptions of the noncommutative Balmer spectra of M$\Delta$Cs 
as topological spaces. 
%Although one is interested in such classifications for compact M$\Delta$Cs (e.g. stable module categories of finite dimensional modules of 
%finite dimensional Hopf algebras), one needs to go to noncompact parts for these classification results.

We refer the reader to Sections \ref{tr-cat} and \ref{SS:supportdata} for background on compactly generated M$\Delta$Cs, $\bK$, and their compact parts 
$\bK^c$. 
For $\mc{S} \subseteq \bK^c$, we will denote by $\langle \mc{S} \rangle$ the smallest thick ideal of $\bK^c$ containing $\mc{S}$. 

\begin{theorem} \label{T:theoremC} Let $\bK$ be a compactly generated M$\Delta$C and $\sigma : \bK \to \XX$ be an extended weak support datum
for a Zariski space $X$ such that $\Phi(\langle C \rangle)$ is closed for every compact object $C$. Assume that $\sigma$ satisfies the faithfulness property \eqref{EE:supportone} and the 
realization property \eqref{EE:supporttwo}. 
\begin{enumerate}
\item[(a)] The map 
\[
\Phi_\sigma : \ThickId(\bK^c) \to \mc{X}_{sp}
\]
is an isomorphism of ordered monoids, where the set of thick ideals of $\bK^c$ is equipped with the operation $\bI, \bJ \mapsto \langle \bI \otimes \bJ \rangle$ and the inclusion partial order, 
and $\mc{X}_{sp}$ is equipped with the operation of intersection and the inclusion partial order.
\item[(b)] The universal map $f_\sigma: X \to \Spc \bK^c$ from Theorem~\ref{T:theoremB}(b), given by 
\[
f_\sigma(x)= \{ A \in \bK^c : x \not \in \Phi_{\sigma} (\langle A \rangle) \}
\quad
\mbox{for} 
\quad x \in X,
\]
is a homeomorphism.
\end{enumerate}
\end{theorem} 

Our third kind of support datum, which we name {\it{quasi support datum}}, is a map $\sigma : \bK \to \mc{X}$ that has the properties (i)-(iv) from the previous subsection, 
together with the additional property 
\smallskip
\begin{enumerate}
\item[(v'')] $\sigma(A \otimes B)\subseteq \sigma (A)$, for all $A, B \in \bK$. 
\end{enumerate}
\smallskip
In ideal-theoretic terms, this property is equivalent to
\smallskip
\begin{enumerate}
\item[(v'')] $\Phi_\sigma ( \llangle A \rrangle_{\mathrm{r}} ) = \sigma (A)$, $\forall A \in \bK$,
\end{enumerate}
\smallskip
where $\llangle A \rrangle_{\mathrm{r}}$ denotes the smallest thick right ideal of $\bK$ containing $A$. 
(The double bracket notation is used to distinguish thick ideals of $\bK$ from those of $\bK^c$ used in 
Theorem \ref{T:theoremC}, see Sections \ref{SS:noncomm-Hopkins}--\ref{S:onesided} for further details.)
We can now summarize 
our classification results for right ideals for the compact objects in a compactly generated M$\Delta$C. 

\begin{theorem} \label{T:theoremD}
\begin{enumerate}
\item[(a)]  Let $\bK$ be a compactly generated M$\Delta$C and $\sigma:\bK \to \mc{X}$ be an assignment that satisfies properties (i), (ii'), (iii) and (iv) for extended support datum, and such that 
$\sigma$ when restricted to $\bK^{c}$ is a quasi support datum for a Zariski space $X$ satisfying the realization property (7.1.\ref{E:supportthree}) and the homological 
Assumption~\ref{A:keyproj}. Then the map
\[
\Phi_\sigma : \{\text{thick right ideals of $\bK^{c}$}\} \to  \mc{X}_{sp}
\]
is a bijection.
\item[(b)]  Let $A$ be a finite-dimensional Hopf algebra over a field $\kk$ that satisfies the standard Assumption~\ref{A:fg} for finite generation and the homological 
Assumption~\ref{A:keyproj}. Set 
%$\bK^{c}=\stmod(A)$ and 
$X=\operatorname{Proj}(\operatorname{H}^{\bullet}(A,\kk))$. 
The standard cohomological support $\sigma : \stmod(A) \to \mc{X}_{cl}$ 
is a quasi support datum, and as a consequence, there is a bijection
$$
\Phi_\sigma: \{\text{thick right ideals of $\stmod(A)$}\} \to \XX_{sp},
$$
where $\stmod(A)$ is the stable (finite-dimensional) module category of $A$. 
\end{enumerate}
\end{theorem}

The maps in both parts are isomorphisms of ordered posets if and only if the quasi support datum $\sigma$ is a support datum.
Section \ref{smallq-Borel} illustrates how the classification of thick right ideals of an M$\Delta$C from Theorem~\ref{T:theoremD} can be converted to 
a classification of its thick two-sided ideals and then to a description of its Balmer spectrum.

In noncommutative ring theory, general classification results for prime spectra and sets of right ideals like the ones in Theorem~\ref{T:theoremC} and ~\ref{T:theoremD} are not available. 
Even worse, the Zariski topology of the prime spectrum of a noncommutative ring is very rarely known. For instance, in the case of quantum groups prime ideals are classified
\cite{Joseph} but the problem for describing the Zariski topology is wide open with only a conjecture  \cite{BG}: 
the inclusions between primes are unknown and only the maximal ideals have been classified \cite{Yakimov}.

It is unlikely that the Balmer spectrum of each M$\Delta$C is a Zariski space. Therefore, it would be desirable to find extensions of Theorem~\ref{T:theoremC} and \ref{T:theoremD} 
with weaker conditions.
%%%%%%%%%%%
\subsection{Applications} In Sections \ref{smallq-Borel} and \ref{BW} we illustrate Theorem~\ref{T:theoremC} and \ref{T:theoremD} by giving classifications 
of the noncommutative Balmer spectra, 
thick two-sided/right ideals for the stable module categories of all small quantum groups for Borel subalgebras and classifications of the noncommutative Balmer spectra and thick two-sided ideals of the Benson--Witherspoon Hopf algebras.
The approach to the Balmer spectrum through Theorem~\ref{T:theoremD} is applicable to the first case and the one through Theorem~\ref{T:theoremC} to the second case.
The Benson--Witherspoon Hopf algebras are the Hopf duals of smash products of a (finite) group algebra and a coordinate ring of a group. 

Sections \ref{smallq-Borel} and \ref{BW} also illustrate that the three notions of noncommutative support data that we define are distinct. 
This differs from the commutative situation where under natural assumptions a quasi support datum is a support datum \cite[Proposition 2.7.2]{BKN2}. 
%\medskip
%$\\
%\noindent
\subsection{Acknowledgements.} The authors would like to thank Cris Negron and Julia Pevtsova for 
providing access to a preprint on their work on the support varieties via non-commutative hypersufaces 
\cite{NP}. Discussions with Cris Negron provided valuable insights for us to refine the results in Section~\ref{smallq-Borel} and to 
add Section~\ref{S:actionpi}. We also acknowledge Henning Krause for useful discussions about prior work on support theory and ideals in tensor triangular geometry.

\section{Triangulated 2-Categories}
\label{two}
%%%%%%%%
\subsection{Triangulated categories and compactness} 
\label{tr-cat}
For the definition and properties on triangulated categories we refer the reader to \cite[Section 1.3]{BIK2} or \cite{Neeman2}.    
Let $\bT$ be a triangulated category with shift $\Sigma:\bT\rightarrow \bT$. Since $\bT$ is triangulated one has a set of distinguished triangles: 
$$M\rightarrow N \rightarrow Q \rightarrow \Sigma M.$$  

An additive subcategory $\bS$ of $\bT$ is a {\em triangulated subcategory} if (i) $\bS$ is non-empty and full, 
(ii) for $M\in \bS$, $\Sigma^{n}M \in \bS$ for all $n\in {\mathbb Z}$, and (iii) if $M\rightarrow N \rightarrow Q \rightarrow \Sigma M$ is 
distinguished triangle in $\bT$ and if two objects in $\{M,N,Q\}$ are in $\bS$ then the third is in $\bS$. 
A {\em{thick}} subcategory of $\bT$ is a triangulated subcategory $\bS$ with the property that, 
$M=M_{1}\oplus M_{2} \in \bS$ implies $M_{j}\in \bS$ for $j=1,2$.
%If $\CC$ is a collection of objects of $\bT$, let $\text{Thick}(\CC)$ denote the smallest thick subcategory containing $\CC$.  

In this paper, at some points we will assume that $\bT$ admits set indexed coproducts. 
A {\em localizing subcategory} $\bS$ is a  triangulated subcategory of $\bT$ that is closed under taking set indexed coproducts. Recall that 
localizing subcategories are thick \cite{BIK3}. Given $\CC$ a collection of objects of $\bT$, let $\Loc_{\bT}(\CC)=\Loc (\CC )$ 
be the smallest localizing subcategory containing $\CC$. 

For $A$ and $B$ objects of $\bT$, we will denote the space of morphisms from $A$ to $B$ by $\bT(A,B)$. If $\bT(C,-)$ commutes with set indexed coproducts for an object $C$ in $\bT$ then $C$ is called \emph{compact}. The full subcategory of compact objects in $\bT$ 
are denoted by $\bT^{c}$. We say that $\bT$ is \emph{compactly generated} if the isomorphism classes of compact objects form a set and if for each non-zero $M \in \bT$ there is a compact object $C$ such that $\bT(C,M) \neq 0$.  It turns out that when $\bT$ is compactly generated there exists a set of compact objects, $\CC$, 
such that $\Loc_{\bT}(\CC)=\bT$ (cf. \cite[Proposition 1.47]{BIK2}). 

\subsection{2-Categories and triangulated 2-categories}
\label{tri-two}

Recall that a {\em{2-category}} is a category enriched over the category of categories. This means that a 2-category $\bK$ has the following structure:
\begin{enumerate}
\item[(i)] A collection of objects;
\item[(ii)] For any two objects $A_1$ and $A_2$, a collection of 1-morphisms, denoted $\bK(A_1,A_2)$;
\item[(iii)] For any two 1-morphisms $f, g \in \bK(A_1,A_2)$, a collection of 2-morphisms $f \to g$, denoted $\bK(f,g)$.
\end{enumerate}
The 1- and 2-morphisms are composed in several ways and admit unit objects. They satisfy a list of axioms, among which we single out the following ones:
%This structure must satisfy the following axioms:
\begin{enumerate}
%\item[(iv)] $\bK$ is a category, where the objects are the objects, and the morphisms are the 1-morphisms;
\item[(iv)] $\bK(A_1,A_2)$ is a 1-category, where the objects are the 1-morphisms of $\bK$, and the morphisms are the 2-morphisms of $\bK$; 
\item[(v)] Composition $\bK(A_2, A_3) \times \bK (A_1, A_2) \to \bK (A_1, A_3)$ is a bifunctor. For 1-morphisms 
$f \in \bK(A_2, A_3)$ and $g \in \bK (A_1, A_2)$,
the composition will be denoted by $f \circ g$.  
\end{enumerate}
For details on 2 categories and their role in categorification, we refer the reader to \cite{Lauda1,Mazorchuk2}.

We say that a 2-category $\bK$ is a {\em{triangulated 2-category}} if $\bK(A_1, A_2)$ is a triangulated 1-category for all pairs of objects $A_1$ and $A_2$, and the compositions
\[
\bK(A_2, A_3) \times \bK(A_1, A_2) \to \bK(A_1, A_3)
\]
are exact bifunctors for all objects $A_1, A_2,$ and $A_3$ of $\bK$. 
%In the language of enriched categories, this means that $\bK$ is not just enriched over the category of categories, 
%but is in fact enriched over the category of triangulated categories.

Throughout, we will assume that all triangulated 2-categories $\bK$ which we work with are {\em{small}}. This means that the objects of $\bK$ form a set, the 1-morphisms of $\bK(A_1, A_2)$ form a set for all objects $A_1$ and $A_2$, and the 2-morphisms $\bK(f,g)$ form a set for all 1-morphisms $f$ and $g$.

For a triangulated 2-category $\bK$, we will denote by $\bK_1$ the isomorphism classes of all 1-morphisms of $\bK$, and for any subsets $X, Y$ of $\bK_1$, we denote by $X \circ Y$ the set of all isomorphism classes which have a representative of the form $f \circ g$, where $f$ represents an element of $X$ and $g$ represents an element of $Y$
and $f$, $g$ are composable in this order.

A {\em{monoidal triangulated category}} (M$\Delta$C for short) is a monoidal category $\bT$ in the sense of Definition 2.2.1 of \cite{EGNO1} which is triangulated and for which the monoidal structure 
$\otimes : \bT \times \bT \to \bT$ is an exact bifunctor. 

\bre{monoidal-tri}
In the same way that 2-categories are generalizations of monoidal categories, triangulated 2-categories are generalizations of 
monoidal triangulated categories. From a monoidal triangulated category $\bT$ one can build a triangulated 2-category $\bK$ with one object as follows:
\begin{enumerate}
\item[(a)]The 1-morphisms of $\bK$ are defined to be the objects of $\mc{M}$, and composition of 1-morphisms $f$ and $g$ is given by 
\[
f \circ g := f \otimes g
\]
where the monoidal product is the product in $\bT$;
\item[(b)] The 2-morphisms of $\bK$ are the morphisms of $\bT$.
\end{enumerate} 
In the other direction, given a triangulated 2-category $\bK$ with one object, one defines a monoidal triangulated category $\bT$ in the following manner:
\begin{enumerate}
\item[(c)] Objects are the 1-morphisms of $\bK$;
\item[(d)] The monoidal product between objects is defined by 
\[
f \otimes g := f\circ g
\] 
where the right hand side is composition of 1-morphisms in $\bK$;
\item[(e)] The morphisms of $\bT$ are the 2-morphisms of $\bK$.
\end{enumerate}
\ere

\subsection{Triangulated 2-categories obtained as derived categories}

The following example gives a natural construction of a triangulated 2-category.

\bex{derived}
Let $\mc{S}$ be a set of (noncommutative) algebras. Then we define the 2-category $\bK_{\mc{S}}$ by:
\begin{enumerate}
\item[(a)] Objects are the elements of the set $\mc{S}$;
\item[(b)] The 1-morphisms $A_1 \to A_2$ are all finite-dimensional $A_2$-$A_1$ bimodules, and composition of 1-morphisms is given by tensor product: if $f_1: A_1 \to A_2$, then $f_1$ is an $A_2$-$A_1$ bimodule; likewise, for $f_2: A_2 \to A_3$, then $f_2$ is a $A_3$-$A_2$ bimodule; and then $$f_2 \circ f_1 :=f_2 \otimes_{A_2} f_1;$$
\item[(c)] 2-morphisms $f \to g$ are bimodule homomorphisms.
\end{enumerate}

$\bK_{\mc{S}}$ is a 2-category where each $\bK_{\mc{S}} (A_1, A_2)$ is an abelian 1-category. However, this does not fit to the setting of abelian 2-categories treated in \cite{VY1}
because the composition of 1-morphisms is not exact; instead, in general, it is only right exact. 

Therefore, we consider the triangulated 2-category $\bD_{\mc{S}}^-$, constructed from $\bK_{\mc{S}}$:
\begin{enumerate}
\item[(d)] Objects are the same as in $\bK_{\mc{S}}$, namely the elements of $\mc{S}$;
\item[(e)] The categories $\bD_{\mc{S}}^-(A_1, A_2)$ are the bounded above derived categories of the corresponding abelian categories $ \bK_{\mc{S}}(A_1, A_2)$, and composition 
\[
\bD_{\mc{S}}^-(A_2, A_3) \times \bD_{\mc{S}}^-(A_1,A_2) \to \bD_{\mc{S}}^-(A_1, A_3)
\]
is given by the left derived tensor product: $$f_2 \circ f_1 :=f_2 \otimes^L_{A_2} f_1$$ for $f_1: A_1 \to A_2, f_2: A_2 \to A_3$. 
\end{enumerate}

This gives us a triangulated 2-category by \cite[Exercise 10.6.2]{Weibel1}. 
\eex

\section{The Prime, Semiprime, and Completely Prime Spectrum}

In this section, we extend the notions of prime, semiprime, and completely prime spectra of noncommutative rings to the case of triangulated 2-categories.

\subsection{Thick ideals of a triangulated 2-category} We start with some terminology for 2-categories.
%Recall that a thick subcategory $\bJ$ of a triangulated category $\bT$ 
%is a subcategory which is closed under direct summands and satisfies the two-out-of-three condition (if $A \to B \to C \to \Sigma A$ is a distinguished 
%triangle of $\bT$ with two objects in $\bJ$, then the third object is also in $\bJ$).
\bde{thick-id}
\begin{enumerate} 
\item[(a)] A {\em{weak subcategory}} of a 2-category $\bK$ is a subcollection of objects $\bI$ of $\bK$ and a collection of subcategories $\bI(A_1, A_2)$ 
of $\bK(A_1, A_2)$, for all $A_1, A_2 \in \bI$, such that composition in $\bK$ restricts to a bifunctor 
\[
\bI(A_2, A_2) \times \bI(A_1, A_2) \to \bI(A_1, A_3)
\] for all $A_1, A_2, A_3 \in \bI$.
\item[(b)] A {\em{thick weak subcategory}} of a triangulated 2-category $\bK$ is a weak subcategory $\bI$ of $\bK$ which has the same objects as $\bK$, and for any pair of objects $A_1, A_2 \in \bI$, the categories $\bI(A_1, A_2)$ are thick subcategories of $\bK(A_1,A_2)$. 
\item[(c)] A {\em{thick ideal}} of a triangulated 2-category $\bK$ is a thick weak subcategory $\bI$ such that for any 1-morphism $f$ in $\bK$ and any 1-morphism $g$ in $\bI$, 
if $g$ and $f$ are composable in either order, then their composition is in $\bI_1$.
\end{enumerate} 
\ede

For any collection of 1-morphisms $\mc{M}$, we will denote by $\langle \mc{M} \rangle$ the smallest thick ideal containing $\mc{M}$, which exists since the intersection of any family of thick ideals is a thick ideal. 

The following lemma is the primary tool by which we connect classical noncommutative ring theory to the setting of triangulated 2-categories. 

\ble{MTN}For every two collections $\mc{M}, \mc{N} \subseteq \bK_1$ of 1-morphisms of a triangulated 2-category $\bK$, 
\begin{equation}
\label{MTN}
\langle \mc{M} \rangle_1 \circ \langle \mc{N} \rangle_1 \subseteq \langle \mc{M} \circ \bK_1 \circ \mc{N} \rangle_1.
\end{equation}
\ele
\begin{proof}
First, we will show that 
\begin{equation}
\label{first-emb}
\langle \mc{M} \rangle_1 \circ \mc{N} \subseteq \langle \mc{M} \circ \bK_1 \circ \mc{N} \rangle_1.
\end{equation}
Let $\bI$ denote the collection of all 1-morphisms $f$ which have the property that for all $n \in \mc{N}, t \in \bK_1$, one has $ftn \in \langle \mc{M} \circ \bK_1 \circ \mc{N} \rangle.$ Note that $\mc{M} \subseteq \bI$. We claim that $\bI$ is a thick ideal.

(1) Suppose that we have a distinguished triangle 
\[
f \to g \to h \to \Sigma f
\]
such that two of $f, g,$ and $h$ are in $\bI$. Since composition is an exact functor, for any $t \in \bK, n \in \mc{N}$, 
\[
ftn \to gtn \to htn \to \Sigma ftn
\]
is a distinguished triangle, and by assumption two out of three of its components are in $\langle \mc{M} \circ \bK \circ \mc{N} \rangle$. Since it is an ideal, so is the third. Therefore, $\bI$ is a triangulated weak subcategory.

(2) Again, let $t \in \bK$ and $n \in \mc{N}$. Suppose $f= g \oplus h$ is in $\bI$. Then $gtn \oplus htn$ is in $\langle \mc{M} \circ \bK \circ \mc{N} \rangle$; by its thickness, $gtn$ and $htn$ are in $\langle \mc{M} \circ \bK \circ \mc{N} \rangle$. Hence, $g$ and $h$ are both in $\bI$. Therefore, $\bI$ is a thick weak subcategory. 

(3) Let $f \in \bI$, and let $g$ and $h$ be 1-morphisms of $\bK$ such that the compositions $gf, fh$ are defined. Then for any $t \in \bK, n \in \mc{N}$, 
$gftn \in \langle \mc{M} \circ \bK \circ \mc{N} \rangle$ by the fact that $\langle \mc{M} \circ \bK \circ \mc{N} \rangle$ is an ideal and $ftn$ is in it; 
and $fhtn \in \langle \mc{M} \circ \bK \circ \mc{N} \rangle$ by the fact that $f \in \bI$ and $ht$ is a 1-morphism of $\bK$. Therefore, $\bI$ is a thick ideal of $\bK$. 

Since $\bI$ is a thick ideal containing $\mc{M}$, $\langle \mc{M} \rangle \subseteq \bI$. From this, we obtain \eqref{first-emb}. 

By symmetry, we can obtain
\begin{equation}
\label{second-emb}
\mc{M} \circ \langle \mc{N} \rangle_1 \subseteq \langle \mc{M} \circ \bK \circ \mc{N} \rangle.
\end{equation}
Then, by an identical argument to (1)-(3) but using instead $\bI$ to be the set of morphisms $f$ for which 
$ftn \in \langle \mc{M} \circ \bK \circ \mc{N} \rangle$ for all $t \in \bK_1, n \in \langle \mc{N} \rangle_1$. This completes the proof. 
\end{proof}

\subsection{Prime ideals of a triangulated 2-category} We now introduce the key notion of prime ideal and provide equivalent 
formulations on when an ideal is prime. In the case of a monoidal triangulated category, this notion was introduced in \cite{BKS}.

\bde{prime-id}
Let $\bP$ be a proper thick ideal of $\bK$. Then $\bP$ a {\em{prime ideal}} of $\bK$ if for every pair of thick ideals $\bI$ and $\bJ$ of $\bK$, 
\[
\bI_1 \circ \bJ_1 \subseteq \bP_1 \quad \Rightarrow \quad \bI \subseteq \bP \; \;  \mbox{or} \; \;  \bJ \subseteq \bP. 
\]
The set of all prime ideals $\bP$ of a triangulated 2-category $\bK$ will be denoted by $\Spc(\bK)$. 
\ede

\bth{prime-equiv}
Suppose $\bP$ is a proper thick ideal of a triangulated 2-category $\bK$. Then the following are equivalent:
\begin{enumerate}
\item[(a)] $\bP$ is prime;
\item[(b)] For all $m, n \in \bK_1$, $m \circ \bK_1 \circ n \subseteq \bP_1$ implies that either $m$ or $n$ is in $\bP_1$;
\item[(c)] For every pair of right thick ideals $\bI$ and $\bJ$ of $\bK$, $\bI_1 \circ \bJ_1 \subseteq \bP_1$ implies that either $\bI$ or $\bJ$ is contained in $\bP$;
\item[(d)] For every pair of left thick ideals $\bI$ and $\bJ$ of $\bK$, $\bI_1 \circ \bJ_1 \subseteq \bP_1$ implies that either $\bI$ or $\bJ$ is contained in $\bP$;
\item[(e)] For every pair of thick ideals $\bI$ and $\bJ$ of $\bK$ which properly contain $\bP$, $\bI_1 \circ \bJ_1 \not \subseteq \bP_1$.
\end{enumerate}
\eth

\begin{proof}
It is clear that just by definition (c)$\Rightarrow$(a), (d)$\Rightarrow$(a), and (a)$\Rightarrow$(e).

(a)$\Rightarrow$(b) Suppose $\bP$ is prime and $m \circ \bK_1 \circ n \subseteq \bP_1$. 
By \leref{MTN}, $\langle m \rangle \circ \langle n \rangle \subseteq \bP$, and thus either $\langle m \rangle$ or $\langle n \rangle \subseteq \bP$; therefore, either $m$ or $n$ is in $\bP$. 

(b)$\Rightarrow$(c) Suppose $\bI, \bJ$ right thick ideals and $\bI_1 \circ \bJ_1 \subseteq \bP_1$, and suppose that neither $\bI$ nor $\bJ$ is contained in $\bP$. Then there exist morphisms $f \in \bI, g \in \bJ$ such that neither $f$ nor $g$ is in $\bP$. Since $\bI$ is a right ideal, $f \circ \bK_1 \subseteq \bI$, and therefore $f \circ \bK_1 \circ g \subseteq \bP_1$. Since neither $f$ nor $g$ is in $\bP$, we have proved the contrapositive.  The direction (b)$\Rightarrow$(d) is analogous.

(e)$\Rightarrow$(b) Let $m \circ \bK_1 \circ n \subseteq \bP_1$. Suppose neither $m$ nor $n$ is in $\bP$. Then by \leref{MTN}
\[
\langle \bP_1 \cup \{m\} \rangle_1 \circ \langle \bP_1 \cup  \{n\} \rangle_1 \subseteq \langle (\bP_1 \cup \{m\}) \circ \bK_1 \circ (\bP_1 \cup \{n\}) \rangle_1 \subseteq \bP_1.
\]
However, both $\langle \bP_1 \cup \{m\} \rangle$ and $\langle \bP_1 \cup \{n\} \rangle$ are thick ideals properly containing $\bP$, thus proving the contrapositive. 
\end{proof}

Recall that a {\em{multiplicative set}} $\mc{M}$ of 1-morphisms in a 2-category is a set of non-zero 1-morphisms contained in $\bK(A,A)$ for some object $A$ such that for all $f, g \in \mc{M}$, 
we have $f \circ g \in \mc{M}$. 

\bth{maximal} Suppose $\mc{M}$ is a multiplicative subset of $\bK_1$ for a triangulated 2-category $\bK$, and suppose $\bI$ is a proper thick ideal of $\bK$ which intersects $\mc{M}$ trivially. 
If $\bP$ is a maximal element of the set 
\[
X(\mc{M}, \bI):=\{\bJ \textrm{ a thick ideal of }\bK: \bJ \supseteq \bI, \bJ_1 \cap \mc{M}= \varnothing \},
\]
then $\bP$ is a prime ideal of $\bK$. \eth

\begin{proof} This follows directly from property (e) of \thref{prime-equiv}. \end{proof}

Every chain of ideals in $X(\mc{M}, \bI)$ has an upper bound given by the union of these ideals. By Zorn's Lemma, all sets 
$X(\mc{M}, \bI)$ have maximal elements. 

\bco{nonempty} For every triangulated 2-category $\bK$, $\Spc(\bK)$ is nonempty. \eco

\begin{proof} This follows from \thref{maximal} by taking $\mc{M}$ as the set $\{1_A\}$ for some nonzero fixed object $A$ and $\bI$ as the zero ideal consisting of all zero 1-morphisms. \end{proof}

\subsection{The Zariski topology} With our notion of primeness, we can now define a topology on the the set of prime ideals. 

\bde{zariski}
Define a map $V$ from subsets of $\bK_1$ to subsets of $\Spc (\bK)$ by $$V(\mc{S})=\{ \bP \in \Spc \bK : \bP \cap \mc{S}= \varnothing \}.$$ 
The {\em{Zariski topology}} on $\Spc(\bK)$ is defined as the topology generated by the closed sets $V(\mc{S})$ for any subset $\mc{S}$ of $\bK_1$.
\ede
We call the topological space $\Spc(\bK)$ the {\em{Balmer spectrum}} of $\bK$. In the case when $\bK$ is a monoidal triangulated category
(where the monoidal structure is symmetric, or more generally, braided), this reduces to the spectrum defined in \cite{Balmer1}. In that setting 
the notions of prime ideal and completely prime ideal coincide. For general monoidal triangulated categories the topology was introduced in \cite{BKS}.

\bre{zariski-monoid-2cat}
In the case that $\bK$ is a monoidal triangulated category, i.e., a triangulated 2-category with one object, all closed sets under the Zariski topology are of the form $V(\mc{S})$ described above. In that case, it is easy to see that $V(0)=\varnothing, V(1)=\Spc( \bK), V(\mc{S})_1 \cup V(\mc{S}_2)= V( \mc{S}_1 \oplus \mc{S}_2)$, and $\bigcap V(\mc{S}_i)=V \left ( \bigcup \mc{S}_i \right ).$ However, if $\bK$ has more than one object, there is no longer any set $\mc{S}$ such that $V(\mc{S})= \Spc (\bK)$, 
and, in general, the union $V(\mc{S}_1) \cup V(\mc{S}_2)$ cannot be written as $V(\mc{S})$ for some $\mc{S}$.
\ere

\bre{mistake-BKS} It was stated in \cite[Lemma 8.2]{BKS} that  $\Spc \bK$ is functorial for a monoidal triangulated category $\bK$. The proof given there appears to have a gap. In the proof it is stated but not justified that the Balmer support composed with a monoidal triangulated functor is again a support datum.
We do not expect such a functoriality; the case of the spectrum of a noncommutative ring is well known to be non-functorial.
\ere

\subsection{Semiprime ideals of a triangulated 2-category} In this section we define semiprime ideals in an abelian 2-category and study their properties.

\bde{semip} A thick ideal of an abelian 2-category will be called {\em{semiprime}} if it is an intersection of prime ideals, cf. \cite{BKS}.
\ede

\bth{semiprime-equiv}
Suppose $\bQ$ is a proper thick ideal of a triangulated 2-category $\bK$. Then the following are equivalent:
\begin{enumerate}
\item[(a)] $\bQ$ is a semiprime ideal;
\item[(b)] For all $f \in \bK_1$, if $f \circ \bK_1 \circ f \subseteq \bQ_1$, then $f \in \bQ_1;$
\item[(c)] If $\bI$ is any thick ideal of $R$ such that $\bI_1 \circ \bI_1 \subseteq \bQ_1$, then $\bI \subseteq \bQ$;
\item[(d)] If $\bI$ is any thick ideal properly containing $\bQ$, then $\bI_1 \circ \bI_1\not \subseteq \bQ_1$;
\item[(e)] If $\bI$ is any left thick ideal of $\bK$ such that $\bI_1 \circ \bI_1 \subseteq \bQ_1$, then $\bI \subseteq \bQ$.
\item[(f)] If $\bI$ is any right thick ideal of $\bK$ such that $\bI_1 \circ \bI_1 \subseteq \bQ_1$, then $\bI \subseteq \bQ$.
\end{enumerate}
\eth

\begin{proof}
(a)$\Rightarrow$(b) Suppose $f \circ \bK_1 \circ f \subseteq \bQ_1$, and let $\bQ= \bigcap_{\alpha} \bP_\alpha$ for prime ideals $\bP_{\alpha}$. Then by \thref{prime-equiv}, $f$ is in $\bP_\alpha$ for each $\alpha$, and hence $f \in \bQ$.

(b)$\Rightarrow$(e) Suppose $\bI_1 \circ \bI_1 \subseteq \bQ_1$, and suppose $\bI\not \subseteq \bQ$. Then there is $f \in \bI$ with $f \not \in \bQ$. Hence, since $ tf \in \bI$ for each $t \in \bK_1$, we have $\bK_1 \circ f \subseteq \bI$, and hence $f \circ \bK_1 \circ f \subseteq \bQ_1$. Since $f \not \in \bQ_1,$ $\bQ$ does not satisfy property (b). The implication (b)$\Rightarrow$(f) is analogous.

The implications (e)$\Rightarrow$(c) and (f)$\Rightarrow$(c) are clear, as is (c)$\Rightarrow$(d). 

(d)$\Rightarrow$(a) Let $\bQ$ a proper thick ideal satisfying (d), and let $\bR$ be the semiprime ideal defined as the intersection of all prime ideals containing $\bQ$; there is at least one such prime ideal by \thref{maximal}. We will show that $\bR =\bQ$; to do this, for an arbitrary 1-morphism $f$ which is not in $\bQ_1$, we will produce a prime ideal which contains $\bQ$ and does not contain $f$. Denote $f:=f_1.$ Since $f_1 \not \in \bQ_1,$ we have $\bI^{(1)}:=\langle \bQ_1 \cup\{f_1\} \rangle$ properly contains $\bQ$. Hence, there is some $f_2 \in \bI^{(1)} \circ \bI^{(1)}$ with $f_2 \not \in \bQ$. Continue in this manner, defining $\bI^{(i)}:=\langle \bQ_1 \cup \{f_i\} \rangle$ and then $f_{i+1}$ as an element of $\bI^{(i)} \circ \bI^{(i)}$ which is not in $\bQ_1$. Note that for any $i$, 
\[
\bI^{(i)} \subseteq \bI^{(i-1)} \circ \bI^{(i-1)} \subseteq \bI^{(i-1)}.
\]

 Now consider a maximal element of the set of ideals $\bI$ such that $\bQ \subseteq \bI$ and $f_i \not \in \bI_1$ for all $i$. Call this maximal element $\bP$. We will demonstrate that $\bP$ is prime. Consider $\bJ$, $\bK$ two ideals properly containing $\bP$. By maximality of $\bP$, each of $\bJ$ and $\bK$ contain one of the $f_i$. Let $f_j \in \bJ$, $f_k \in \bK$. Without loss of generality, let $j \geq k$. Then since $\bK$ contains both $\bQ$ and $f_k$, 
\[
\bK \supseteq \bI^{(k)} \supseteq \bI^{k+1} \supseteq... \supseteq \bI^{(j)}.
\] 
Therefore, $f_j$ is in both $\bJ$ and $\bK$. Then note that by \leref{MTN}, 
\[
f_{j+1} \in \bI^{(j)} \circ \bI^{(j)} \subseteq \langle (\bQ_1 \cup \{f_j\}) \circ \bK_1 \circ(\bQ_1 \cup \{f_j\}) \rangle
\]
and 
\[
f_{j+1} \not \in \bP_1.
\]
Therefore,
\[
f_j \circ \bK_1 \circ f_j \not \subseteq \bP_1,
\]
which implies
\[ 
\bJ_1 \circ \bK_1 \not \subseteq \bP_1.
\]
Thus, by \thref{prime-equiv}, $\bP$ is prime. By construction, it contains $\bQ$ and not $f_1$, which completes the proof.
\end{proof}

\subsection{Completely prime ideals of a triangulated 2-category} In this section we introduce the notion of completely prime ideals 
for a triangulated 2-category and show how these ideals arise via categories of 1-morphisms. 

\bde{compl}
A thick ideal $\bP$ of a triangulated 2-category $\bK$ will be called {\em{completely prime}} when it has the property that 
for all $f, g \in \bK_1$:
\[
f \circ g \in \bP_1 \quad \Rightarrow \quad f \in \bP_1 \; \; \mbox{or} \; \; g \in \bP_1.
\]
\ede
We have the following lemma, whose proof is direct and left to the reader.
\ble{c-prime} Let $\bK$ be a triangulated 2-category.
\begin{enumerate}
\item[(a)] If $\bP$ is a completely prime ideal of $\bK$, then there exists an object $A$ of $\bK$ and 
a completely prime ideal $\bQ$ of the triangulated 2-category with 1-object $\bK_A$ such that
\begin{equation}
\label{PQ}
\bP(B,C) =
\begin{cases}
\bQ(A,A), &\mbox{if} \; \; B=C =A
\\
\bK(B,C), &\mbox{otherwise}.
\end{cases}
\end{equation}
\item[(b)] If $A$ is an object of $\bK$ and $\bQ$ is a completely prime ideal of $\bK_A$ such that 
\begin{equation}
\label{PQcond}
\bK(B,A) \circ \bK(A,B) \subseteq \bQ(A,A)
\end{equation}
for every object $B$ of $\bK$, then \eqref{PQ} defines a completely prime ideal $\bP$ of $\bK$.
\end{enumerate}
\ele

\section{Noncommutative Support Data}

In this section we develop a notion of (noncommutative) {\em{support datum}} for {\em{monoidal triangulated categories}} $\bK$ and show that 
the Balmer spectrum of $\bK$ provides a {\em{universal final object}}. In parallel to noncommutative ring theory, we are lead to 
convert support data from maps with a domain the set of objects of $\bK$ to maps with a domain given by the (two-sided) thick ideals of $\bK$, 
in which case the maps become simply morphisms of monoids. The second incarnations will play a key role in the rest of the paper. The 
original support datum maps are not uniquely reconstructed from them. To handle this, we develop a notion 
of {\em{weak support datum}} and prove a universality property for the Balmer spectrum in this setting also. 

\subsection{Defining noncommutative support data} 
\label{noncomm-supp}
Assume throughout the rest of the paper that 
\medskip
\begin{center}
{\em{ $\bK$ is a monoidal triangulated category (M$\Delta$C for short) and $X$ is a topological space.}}
\end{center}
\medskip
Let $\mc{X}$ denote the collection of all subsets of $X$, $\mc{X}_{cl}$ 
the collection of all closed subsets of $X$, and $\mc{X}_{sp}$ the collection of all specialization closed subsets of $X$, that is, arbitrary unions of closed sets.

When it is necessary to emphasize the underlying topological space $X$, we will use the notation $\mc{X}_{cl}(X)$ and $\mc{X}_{sp}(X)$. 

\bde{nc-support} \label{supportdatum}
Let $\bK$ be a monoidal triangulated category and $\sigma$ a map $\bK \to \mc{X}$. We will say that $\sigma$ is a (noncommutative) {\em  support datum} if 
the following hold:
\begin{enumerate}
\item[(a)] $\sigma(0)=\varnothing $ and $ \sigma(1)= X$;
\item[(b)] $\sigma(A\oplus B)=\sigma(A)\cup \sigma(B)$, $\forall A, B \in \bK$; 
\item[(c)] $\sigma(\Sigma A)=\sigma(A)$, $\forall A \in \bK$; 
\item[(d)] If $A \to B \to C \to \Sigma A$ is a distinguished triangle, then $\sigma(A) \subseteq \sigma(B) \cup \sigma(C)$;
\item[(e)] $\bigcup_{C \in \bK} \sigma (A   \otimes C \otimes B) = \sigma (A) \cap \sigma(B)$, $\forall A, B \in \bK$. 
\end{enumerate}
\ede
It follows from conditions (a) and (d) that $\sigma$ is constant along the isomorphism classes of objects of $\bK$. The same will be true for all other 
notions of support datum that we consider in this paper.
Recall the map $V$ defined in \deref{zariski}.

\ble{v-support}
For any M$\Delta$C, $\bK$, the restriction to objects of $\bK$ of the map $V$ is a support datum $\bK \to\mc{X}_{cl}(\Spc (\bK))$.
\ele

\begin{proof}
By the definition of the Zariski topology, $V(A)$ is closed for any $A$. We verify the properties (a)-(e) in Definition~\ref{supportdatum} below. 

(a) $V(0)= \varnothing$ because 0 is in every prime ideal of $\bK$. Since prime ideals are required to be proper, 1 is in no prime ideal, and hence $V(1)=\Spc(\bK)$. 

(b) Prime ideals are closed under sums and summands. Hence, if $\bP$ is a prime ideal, then $A \oplus B \in \bP$ if and only if both $A$ and $B$ are in $\bP$. 

(c) Prime ideals are closed under shifts; hence, $A$ is in a prime ideal $\bP$ if and only if $\Sigma A$ is in $\bP$.

(d) Since prime ideals are triangulated, if $$A \to B \to C \to \Sigma A$$ is a distinguished triangle with $\bP \in V(A)$, then $A \not \in \bP$ and hence one of $B$ or $C$ be not be in $\bP$. Therefore, $\bP$ is in $V(B)$ or $V(C)$.

(e) First, we will show $\subseteq$. Suppose $\bP$ is in some $V(A \otimes C \otimes B)$ for some $C$; in other words, $A \otimes C \otimes B \not \in \bP$. Then since $\bP$ is a thick ideal, neither $A$ nor $B$ can be in $\bP$, and hence $\bP \in V(A)$ and $V(B)$. For $\supseteq$, suppose $\bP \in V(A) \cap V(B)$. Then by the primeness condition, $A \otimes \bK \otimes B  \not \subseteq \bP$, since that would imply either $A$ or $B$ would be in $\bP$. Hence, there is some $C$ with $A \otimes C \otimes B \not \in \bP$, and so $\bP \in V(A \otimes C \otimes B)$ for some choice of $C$.\end{proof}

\subsection{The final support datum} We begin with a lemma that shows that if we have continuous maps from $X$ to  $\Spc \bK$ whose inverse 
images agree on closed sets then the maps must be equal. 

\ble{f1-f2-agree}
Let $X$ be a set and $f_1, f_2: X \to \Spc \bK$ be two maps such that $f_1^{-1}( V(A))=f_2^{-1}(V(A))$ for all objects $A$ of $\bK$. Then $f_1=f_2$.
\ele

\begin{proof} By assumption, for all $A \in \bK$ and $x \in X$, $f_1(x) \in V(A) \Leftrightarrow f_2(x) \in V(A)$. Hence, 
for all $x \in X$,
$$\bigcap_{A \in \bK, f_1(x) \in V(A)} V(A) = \bigcap_{A \in \bK, f_2(x) \in V(A)} V(A),$$ 
and thus, 
\begin{align*}
V(\bK \backslash f_1(x) )&= \overline{\{f_1(x)\}} =  \bigcap_{A \in \bK, f_1(x) \in V(A)} V(A) =  \bigcap_{A \in \bK, f_2(x) \in V(A)} V(A)\\
&= \overline{ \{f_2(A)\}} = V(\bK \backslash f_2(x)).
\end{align*}
Since $f_1(x) \in V( \bK \backslash f_1(x))$, the above equality implies that $f_1(x) \in V(\bK \backslash f_2(x))$. 
Therefore, $f_1(x) \subseteq f_2(x)$, and analogously, $f_2(x) \subseteq f_1(x)$. Hence,
$f_1 = f_2$.
\end{proof}

With the prior results we can show that there exists a final support datum. 

\bth{univ-supp}
In the collection of support data $\sigma$ such that $\sigma(A)$ is closed for each object $A$, the support $V$ is the final support object for every M$\Delta$C $\bK$: 
that is, given any other support datum $\sigma$ as above, there is a unique continuous map $f_\sigma: X \to \Spc \bK$ 
satisfying $\sigma(A)= f_\sigma^{-1}( V(A)).$ Explicitly, this map is defined by $$f_\sigma(x)= \{ A \in \bK : x \not \in \sigma(A)\}.$$
\eth

\begin{proof}
The uniqueness of this map follows directly from \leref{f1-f2-agree}. We need to show that the formula given for $f_\sigma(x)$ defines a prime ideal, and that 
$\sigma(A)=f_\sigma^{-1}(V(A))$, which will then imply that $f$ is continuous. 

The subset $f_\sigma(x)$ satisfies the two-out-of-three condition, since if $$A \to B \to C \to \Sigma A$$ is a distinguished triangle with $B$ and $C$ in $f_\sigma(x)$, this means that $x$ is not in $\sigma(A)$ or $\sigma(B)$, and by condition (d) of support data that implies that $x \not \in \sigma(A)$, and so $A \in f_\sigma(x)$. Additionally, $A \in f_\sigma(x)$ if and only if $x \not \in \sigma(A)$, which, by condition (c) for support data, happens if and only if $x \not \in \sigma(\Sigma A)$, i.e. $\Sigma A \in f_\sigma(x)$. Therefore, $f_\sigma(x)$ is closed under shifts, and so it is triangulated. 

The triangulated subcategory $f_\sigma(x)$ is also thick, because if $A \oplus B \in f_\sigma(x)$ then $x$ is not in $\sigma(A \oplus B)$, and by condition (b) 
of support data $x$ is not in $\sigma(A)$ or $\sigma(B)$. Therefore, $A$ and $B$ are in $f_\sigma(x)$. 

Next, we will observe that $f_\sigma(x)$ is a (two-sided) ideal. Suppose that $A \in f_\sigma(x)$. Then $x \not \in \sigma(A)$. For any $B$, 
since by condition (e) for support data $$\sigma(A \otimes B) \subseteq \sigma(A) \cap \sigma(B),$$ we have $x \not \in \sigma(A \otimes B),$ 
and therefore $A \otimes B \in f_\sigma(x)$. The same argument shows that $B \otimes A \in f_\sigma(x)$ as well. 

Lastly, we verify that $f_\sigma(x)$ is prime. 
Suppose $A \otimes \bK \otimes B \subseteq f_\sigma(x)$. Then for all objects $C$, $x \not \in V(A \otimes C \otimes B)$. Hence, by condition (e) of being a support 
datum, $x \not \in \sigma(A) \cap \sigma(B)$, implying that it is not in $\sigma(A)$ or $\sigma(B)$. Therefore, either $A$ or $B$ is in $f_\sigma(x)$. 

Lastly, we just verify the formula $\sigma(A)= f_\sigma^{-1}( V(A)).$ We have
\begin{align*}
x \in f_\sigma^{-1}( V(A)) & \Leftrightarrow f_\sigma(x) = \{ B : x \not \in \sigma(B)\} \in V(A),\\
& \Leftrightarrow A \not \in \{B : x \not \in \sigma(B)\}\\
& \Leftrightarrow x \in \sigma(A).
\end{align*}
This completes the proof. 
\end{proof}

\subsection{The map $\Phi_\sigma$} \label{SS:phimap}
For any map $\sigma: \bK \to \mc{X}$ with a topological space $X$, we associate a map $\Phi_{\sigma}$ from subsets of $\bK$ to $\mc{X}$ given by
\begin{equation}
\label{Phi}
\Phi_{\sigma} (\mc{S}):= \bigcup_{A \in \mc{S}} \sigma(A).
\end{equation}
By definition, the map $\Phi_\sigma$ is order preserving with respect to the inclusion partial order.

If $\sigma: \bK \to \mc{X}_{cl}$ is a support datum, then $\Phi_{\sigma} (\mc{S})$ is a specialization-closed subset of $X$ for every $\mc{S} \subseteq \bK$. 
We can now prove that the map $\Phi_{\sigma}$ respects the tensor product property on ideals. 
%{\em{When $\sigma$ is obvious from context, we will omit the subscript, writing $\Phi$ instead of $\Phi_{\sigma}$.}} 

\ble{phi-monoid}
Let $\bK$ be an M$\Delta$C and $\sigma: \bK \to \mc{X}_{cl}$ be a support datum.
Then 
$$
\Phi_{\sigma} (\bI \otimes \bJ) = \Phi_{\sigma} (\bI) \cap \Phi_{\sigma}(\bJ)
$$
for every two thick ideals $\bI$ and $\bJ$ of $\bK$, recall \eqref{Phi}.
\ele

\begin{proof}
We have
\begin{align*}
\Phi_{\sigma}( \bI \otimes \bJ)&= \bigcup_{A \in \bI, B \in \bJ} \sigma (A \otimes B) = \bigcup_{A \in \bI, B \in \bJ, C \in \bK} \sigma (A \otimes C \otimes B)\\
&= \bigcup_{A \in \bI, B \in \bJ} \sigma (A) \cap \sigma(B) = \left( \bigcup_{A \in \bI} \sigma(A) \right) \cap \left ( \bigcup_{B \in\bJ} \sigma(B)\right )\\
&= \Phi_{\sigma}( \bI) \cap \Phi_{\sigma}( \bJ).
\end{align*}
\end{proof}

\ble{gamma-ideal}
Let $\bK$ be an M$\Delta$C and $\sigma: \bK \to \mc{X}$ be a support datum. 
For any subset $\mc{S}$ of $\bK$, $\Phi_{\sigma} ( \mc{S}) = \Phi_{\sigma} \left ( \langle \mc{S} \rangle \right ).$
\ele

\begin{proof}
We will check that by adjoining direct summands, shifts, cones, and tensor products to $\mc{S}$ one does not alter $\Phi_{\sigma}(\mc{S})$; this will prove the statement.

Let $\bigoplus_{i \in I} M_i \in \mc{S}$. By condition (b) of support data, $\sigma(M_i) \subseteq \sigma(M \oplus N)$, so adjoining each $M_i$ to $\mc{S}$ does not change $\bigcup_{A \in \mc{S}} \sigma (A)$. 

Let $M \in \mc{S}$. Then, by condition (c) for support data, $\sigma ( \Sigma^m M)= \sigma(M)$, so adjoining shifts to $\mc{S}$ does not alter $\Phi_{\sigma}(\mc{S})$ either.

If $A \to B \to C \to \Sigma A$ is a distinguished triangle with $B$ and $C$ in $\mc{S}$ then $\sigma(A) \subseteq \sigma(B) \cup \sigma(C)$ by condition (d) for support data, so adding $A$ to $\mc{S}$ does not change $\Phi_{\sigma}(\mc{S})$.

Lastly, if $M \in \mc{S}$, then by condition (5) for support data we have $\sigma(M \otimes N) \subseteq \sigma(M) \cap \sigma(N) \subseteq \sigma(M)$. Hence, we can add $M \otimes N$ to $\mc{S}$ without affecting $\sigma(\mc{S})$. Likewise for $N \otimes M$.

Therefore, closing $\mc{S}$ under summands, shifts, cones, and tensor product with arbitrary objects of $\bK$ does not alter $\Phi_{\sigma}$, which proves the lemma.
\end{proof}

The following theorem summarizes our results for support data. 

\begin{theorem} \label{phi-ideal}
For an M$\Delta$C, $\bK$, and a support datum $\sigma: \bK \to \mc{X}_{cl}$, the map
$\Phi_{\sigma}$ is a morphism of ordered monoids from the set of thick ideals of $\bK$ with the operation 
$\bI, \bJ \mapsto \langle \bI \otimes \bJ \rangle$ and the inclusion partial order 
to $\mc{X}_{sp}$ with the operation of intersection and the inclusion partial order. 
\end{theorem} 
\begin{proof} Clearly, $\Phi_\sigma$ preserves inclusions. For every two thick ideals ideals $\bI$ and $\bJ$ of $\bK$, 
\[
\Phi_{\sigma}( \langle \bI \otimes \bJ \rangle ) = \Phi_{\sigma}(\bI) \cap \Phi_{\sigma}(\bJ),
\]
which follows from Lemmas \ref{lphi-monoid} and \ref{lgamma-ideal}.
\end{proof}

\subsection{Weak support data} \label{SS:weaksupport} We now replace condition (e) of support datum with the 
tensor product property on ideals to define the notion of weak support data. 

\bde{nc-w-support}
Let $\bK$ be a monoidal triangulated category and $\sigma$ a map $\bK \to \mc{X}$
. We will call $\sigma$ a (noncommutative) {\em weak support datum} if 
\begin{enumerate}
\item[(a)] $\sigma(0)=\varnothing$ and $ \sigma(1)= X$;
\item[(b)] $\sigma(A\oplus B)=\sigma(A)\cup \sigma(B)$, $\forall A, B \in \bK$; 
\item[(c)] $\sigma(\Sigma A)=\sigma(A)$, $\forall A \in \bK$; 
\item[(d)] If $A \to B \to C \to \Sigma A$ is a distinguished triangle, then $\sigma(A) \subseteq \sigma(B) \cup \sigma(C)$;
\item[(e)] $\Phi_{\sigma}(  \bI \otimes \bJ  ) = \Phi_{\sigma}(\bI) \cap \Phi_{\sigma}(\bJ)$ for all thick ideals $\bI$ and $\bJ$ of $\bK$
\end{enumerate}(recall \eqref{Phi}).
\ede
Note that, for any weak support datum $\sigma : \bK \to \mc{X}$ satisfying the additional condition that each $\Phi_\sigma (\langle A \rangle)$ is closed for any object $A$, and all thick ideals $\bI$ of $\bK$, 
$$\Phi_{\sigma}(\bI)= \bigcup_{A \in \bI} \Phi_{\sigma}(\langle A \rangle) \in \mc{X}_{sp}.$$ By \leref{phi-monoid}, a support datum is automatically a weak support datum. 

The following two lemmas provide information on ideals generated by objects in the category $\bK$. 

\ble{ideal-bi-bj-intersect}
If $\sigma : \bK \to \mc{X}$ is a weak support datum for an M$\Delta$C, $\bK$, and $\bI$ and $\bJ$ are thick ideals of $\bK$, then
$$\Phi_{\sigma}(\langle \bI \otimes \bJ \rangle)= \Phi_{\sigma} (\bI\otimes \bJ ) = \Phi_{\sigma} (\bI) \cap \Phi_{\sigma} (\bJ).$$
\ele

\begin{proof}
By assumption, $\Phi_{\sigma} (\bI \otimes \bJ) = \Phi_{\sigma} (\bI) \cap \Phi_{\sigma} (\bJ)$. Since every element of the set $\bI \otimes \bJ$ is in $\bI$, and in $\bJ$, 
we have $\langle \bI \otimes \bJ \rangle \subseteq \bI \cap \bJ$. Hence $\Phi_{\sigma} (\langle \bI \otimes \bJ \rangle) \subseteq \Phi_{\sigma} (\bI) \cap \Phi_{\sigma} (\bJ).$ 
It is also automatic that $\Phi_{\sigma} ( \bI \otimes \bJ) \subseteq \Phi_{\sigma} (\langle \bI \otimes \bJ \rangle)$. Hence, we have the commutative diagram
\begin{center}
\begin{tikzcd}
\Phi_{\sigma} (\bI \otimes \bJ) \arrow[rd, hook] \arrow[rr] &                                                        & \Phi_{\sigma} (\bI) \cap \Phi_{\sigma} (\bJ) \arrow[ll, "="] \\
                                                  & \Phi_{\sigma} (\langle \bI \otimes \bJ \rangle) \arrow[ru, hook] &                                            
\end{tikzcd}
\end{center}
which gives the statement of the lemma.
\end{proof}

\ble{dist-triang-weak}
Suppose $A \to B \to C \to \Sigma A$ is a distinguished triangle in an M$\Delta$C, $\bK$, and $\sigma$ a weak support datum. Then
$\Phi_{\sigma} ( \langle A \rangle )  \subseteq \Phi_{\sigma} ( \langle B \rangle) \cup \Phi_{\sigma} (\langle C \rangle).$
\ele

\begin{proof}
Define $$\bI= \{ M \in \langle A \rangle : \Phi_{\sigma} ( \bK \otimes M \otimes \bK) \subseteq \Phi_{\sigma} (\langle B \rangle) \cup \Phi_{\sigma} (\langle C \rangle)\}.$$
We will show that $\bI$ is a thick ideal which contains $A$; since it is contained in $\langle A \rangle$ by definition, it is therefore equal to $\langle A \rangle$. 

Suppose $M$ is in $\bI$, and let $X$ and $Y$ two arbitrary objects of $\bK$. We have 
$X \otimes \Sigma M \otimes Y \cong \Sigma ( \Sigma^{-1}(X) \otimes M \otimes \Sigma^{-1}(Y)),$ 
and hence $\sigma ( X \otimes \Sigma M \otimes Y)= \sigma( \Sigma^{-1}(X) \otimes M \otimes \Sigma^{-1}(Y)) \subseteq \Phi_{\sigma} (\langle B \rangle) \cup \Phi_{\sigma} (\langle C \rangle)$, 
showing that $\Sigma(M) \in \bI$. 

Let $K \to L \to M \to \Sigma K$ be a distinguished triangle with $L$ and $M$ in $\bI$. Then by the exactness of the tensor product,
$X \otimes K \otimes Y \to X \otimes L \otimes Y \to X \otimes M \otimes Y \to X \otimes \Sigma K \otimes Y$ is a distinguished triangle.
Hence,
$$\sigma( X \otimes K \otimes Y) \subseteq \sigma( X \otimes L \otimes Y) \cup \sigma(X \otimes M \otimes Y) \subseteq \Phi_{\sigma} (\langle B \rangle) \cup \Phi_{\sigma} (\langle C \rangle).$$
Therefore, $K$ is in $\bI$.

Suppose $M \oplus N$ is in $\bI$. Then $\sigma(X \otimes M \otimes Y) \subseteq \sigma(X \otimes (M \oplus N) \otimes Y) \subseteq \Phi_{\sigma} (\langle B \rangle) \cup \Phi_{\sigma} 
(\langle C \rangle),$ and so $M$ is in $\bI$ (and likewise, so is $N$).

It is clear from the definition of $\bI$ that it is closed under tensoring on the right and left. By exactness of the tensor product, $\bI$ contains $A$. Thus, 
$\bI$ is a thick subideal of $\langle A \rangle$ which contains $A$, and hence, $\bI = \langle A \rangle$.
%Therefore, $\bI=\langle I \rangle$. 
%Thus, for all $M \in \langle A \rangle,$ we have shown that
%$\Phi_{\sigma} (\bK \otimes M \otimes \bK) \subseteq \Phi_{\sigma} (\langle B \rangle) \cup \Phi_{\sigma}(\langle C \rangle)$. 
%In particular, $\sigma(M) \subseteq \Phi_{\sigma} (\langle B \rangle) \cup \Phi_{\sigma}(\langle C \rangle)$. 
%Since this is true for all $M$ in $\langle A \rangle$, we have $$\Phi_{\sigma}(\langle A \rangle) \subseteq \Phi_{\sigma}(\langle B \rangle) \cup \Phi_{\sigma}(\langle C \rangle).$$ 
\end{proof}

\subsection{The final weak support datum} Using the final weak support data one can identify the Balmer spectrum for $\bK$. 

\bth{f-weak}
Suppose that $\bK$ is an M$\Delta$C and $\sigma : \bK \to \mc{X}$ is a weak support datum satisfying the additional condition that $\Phi_{\sigma} (\langle A \rangle)$ is closed for every object $A$ of $\bK$. 
Then there is a unique continuous map $f_\sigma: X \to \Spc \bK$ satisfying $\Phi_{\sigma}(\langle A \rangle)= f_\sigma^{-1}(V(A)),$ for all $A \in \bK$. 
Explicitly, this map is defined by
$$
f_\sigma(x)= \{ A \in \bK : x \not \in \Phi_{\sigma} (\langle A \rangle) \}
\quad
\mbox{for} 
\quad x \in X.
$$
\eth

\begin{proof}
The uniqueness of this map follows directly from \leref{f1-f2-agree}. The continuity will follow from the claimed formula for $f_\sigma^{-1}(V(A))$, since $\Phi_{\sigma} (\langle A \rangle)$ is closed by definition. We need to verify that $f_\sigma(x)$ is a prime ideal, and that $f^{-1}_\sigma(V(A))$ has the formula that has been claimed.

Since $\langle M \rangle = \langle \Sigma M \rangle$, we clearly have $M$ in $f_\sigma(x)$ 
if and only if $\Sigma M$ in $f_\sigma(x)$. If $A \to B \to C \to \Sigma A$ is a distinguished triangle with $B$ and $C$ in $f_\sigma(x)$, 
then by \leref{dist-triang-weak} we have $A$ in $f_\sigma(x)$. If $M \oplus N$ is in $f_\sigma(x)$, then since $M \in \langle M \oplus N \rangle$, we have 
$\Phi_{\sigma} (\langle M \rangle) \subseteq \Phi_{\sigma} (\langle M \oplus N \rangle)$, and so $M \in f_\sigma(x)$ (and likewise for $N$). Suppose 
$M \in f_\sigma(x)$ and $N$ is any object. Then since $\langle M \otimes N \rangle \subseteq \langle M \rangle$, we have 
$\Phi_{\sigma} (\langle M \otimes N \rangle) \subseteq \Phi_{\sigma} (\langle M \rangle)$ and hence $M \otimes N \in f_\sigma(x)$ (and likewise for $N \otimes M$). Hence, $f_\sigma(x)$ is a thick ideal.

Now, suppose that there are thick ideals $\bI$ and $\bJ$ with $\bI \otimes \bJ \subseteq f_\sigma(x)$. Then for each $X \in \bI$ and $Y \in \bJ$, 
$x \not \in \Phi_{\sigma}( \langle X \otimes Y \rangle)$. But now we can observe that
\begin{align*}
\bigcup_{X \in \bI, Y \in \bJ} \Phi_{\sigma} ( \langle X \otimes Y \rangle) & \supseteq \Phi_{\sigma} ( \bI \otimes \bJ) = \Phi_{\sigma} (\bI) \cap \Phi_{\sigma} (\bJ),
\end{align*}
and therefore $x \not \in \Phi_{\sigma} (\bI) \cap \Phi_{\sigma} (\bJ)$. Therefore, one of $\bI$ and $\bJ$ must be in $f_\sigma(x)$. Therefore, $f_\sigma(x)$ is a prime ideal.
Last, we verify that $f_\sigma^{-1}(V(A))= \Phi_{\sigma} (\langle A \rangle)$:
\begin{align*}
f_\sigma^{-1} (V(A)) &= \{ x: f_\sigma(x) \in V(A)\} = \{ x : A \not \in f_\sigma(x)\} \\ 
&= \{x : x \in \Phi_{\sigma} (\langle A \rangle)\} = \Phi_{\sigma} (\langle A \rangle ).
\end{align*}
\end{proof}

\section{A noncommutative Hopkins' Theorem}
\label{SS:noncomm-Hopkins}
In this section we prove a generalization of Hopkins' Theorem which will be used in next section for our first approach to the explicit description of the
(noncommutative) Balmer spectrum of an M$\Delta$C as a topological space.

\subsection{Compactly generated tensor triangulated categories}\label{SS:supportdata} 
We say that a monoidal triangulated category $\bK$ is a {\em{compactly generated}} if 
$\bK$ is closed under arbitrary set indexed coproducts, 
the  tensor  product  preserves  set  indexed  coproducts, $\bK$ is compactly generated as a triangulated category, 
the tensor product of compact objects is compact, $1$ is a compact object, and every compact object is rigid (cf. \cite[Definition 2.10.11]{EGNO1}). 
In particular, $\bK^c$ is an M$\Delta$C on its own. 

In this context, given any subset $\mc{S}$ of $\bK^c$, the notation 
$\langle \mc{S} \rangle$ will refer to the thick two-sided ideal of $\bK^c$ generated by $\mc{S}$, whereas if $\mc{S}$ is any subset of $\bK$, then the notation $\llangle \mc{S} \rrangle$ will refer to the thick two-sided ideal of $\bK$ generated by $\mc{S}$.

Recall that for an M$\Delta$C, $\bK$, and a map $\sigma:\bK \to \mc{X}$, the map $\Phi_{\sigma}$ from subsets of objects of $\bK$ to $\mc{X}$ is defined by 
\eqref{Phi}. At many points in this section and the sequel we will be interested in weak support data 
which satisfy the following two conditions: 
\begin{align} 
&\mbox{ $\Phi_\sigma(\llangle M \rrangle)=\varnothing$ if and only if $M=0$, $\forall M \in \bK$ (Faithfulness Property)}; 
\label{EE:supportone} 
\\
&\mbox{For any $W\in {\mathcal X}_{cl}$ there exists  $M\in \bK^{c}$ such that $\Phi_\sigma(\langle M\rangle)=W$ (Realization} 
\label{EE:supporttwo}
\\
&\mbox{Property).} 
\nonumber
\end{align} 

By rigidity, it follows that there exists an exact contravariant duality functor 
$(-)^{*}:\bK^{c} \rightarrow \bK^{c}$ such that 
\begin{equation} \label{eq:duality}
\bK(N\otimes M, Q)=\bK(N, Q \otimes M^{*})
\end{equation}
for $M\in \bK^{c}$ and $N,Q\in \bK$ \cite[Proposition 2.10.8]{EGNO1}. There are evaluation and coevaluation maps
\[
\ev: V^* \otimes V \to 1,
\]
\[
\coev: 1 \to V \otimes V^*,
\] 
such that the compositions
\begin{equation}
\label{E:ev-coev1}
V \xrightarrow{\coev \otimes \id} V \otimes V^* \otimes V \xrightarrow{ \id \otimes \ev} V
\end{equation}
and 
\begin{equation}
\label{E:ev-coev2}
V^* \xrightarrow{\id \otimes \coev} V^* \otimes V \otimes V^* \xrightarrow{\ev \otimes \id} V^*
\end{equation}
are the identity maps on $V$ and $V^*$, respectively.

\begin{lemma} \label{dual-ideal}
For any object $V$ of an M$\Delta$C,  $\langle V \rangle = \langle V^* \rangle$.
\end{lemma}

\begin{proof}
By the Splitting Lemma for triangulated categories, if a morphism $X \xrightarrow{f} Y$ is a retraction, i.e. there exists a map $Y \xrightarrow{g} X$ such that $g \circ f = \id_X$, then $X$ is a direct summand of $Y$. Since the compositions (\ref{E:ev-coev1}) and (\ref{E:ev-coev2}) are the identity morphisms, $V$ is a direct summand of $V \otimes V^* \otimes V$ and $V^*$ is a direct summand of $V^* \otimes V \otimes V^*$. Hence, $V$ is in $\langle V^* \rangle$, and $V^*$ is in $\langle V \rangle.$
\end{proof}
%\begin{eqnlist} 
%\item \label{E:supportone} $\Phi(\langle M \rangle)=\varnothing$ if and only if $M=0$; 
%\item \label{E:supporttwo} for any $W\in {\mathcal X}_{cl}$ there exists  $M\in \bK^{c}$ such that $\Phi(\langle M\rangle)=W$ (Realization Property).  
%\end{eqnlist} 

\subsection{Localization/colocalization functors} 
\label{loc-fun}
Even though our goal is to classify ideals in a compact M$\Delta$C, we will need to use the localization and colocalization 
functors as given in \cite[Section 3]{BIK1} to construct non-compact objects that will be used in the theory.  In particular, we will employ the following facts stated in \cite[Theorem 3.1.1, Lemma 3.1.2]{BKN1}. 

\begin{theorem} \label{T:localizationtriangles} Let $\bK$ be a compactly generated triangulated category, $\bC$ be a thick subcategory of $\bK^{c}$ and $M$ an object of $\bK$. 
\begin{itemize} 
\item[(a)] There exists a functorial triangle in $\bK$,
$$
\Gamma_{\bC}(M) \to M \to L_{\bC}(M) \to
$$
which is unique up to isomorphism, such that $\Gamma_{\bC}(M)$ is in $\Loc(\bC)$ and there are no non-zero maps in $\bK$ from $\bC$ or, equivalently, from $\Loc(\bC)$ to $L_{\bC}(M)$.
\item[(b)] $M \in \Loc(\bC)$ if and only if $\Gamma_{\bC}(M)\cong M$.
\end{itemize} 
\end{theorem}

\subsection{Hopkins' Theorem} For an object $M\in\bK^{c}$,  recall that $\langle M \rangle \subseteq \bK^{c}$ denotes the thick tensor ideal in $\bK^{c}$ generated by $M$, whereas $\llangle M \rrangle$ will denote the thick tensor ideal in $\bK$ generated by $M$.  
The following result is a generalization of the theorem presented in \cite[Theorem 3.3.1]{BKN1}. Recall the definition of extended weak support datum from Section~\ref{supp}. 
%The proof follows along the same lines and is included to indicate where modifications are needed to deal with case of two-sided ideals. 

\begin{theorem} \label{T:Hopkinstheorem} Let $\bK$ be a compactly generated M$\Delta$C and 
$\sigma : \bK \to \mc{X}$ be an extended weak support datum satisfying the faithfulness condition (\ref{EE:supportone})
for a Zariski space $X$.

Fix an object $M\in\bK^{c}$, and set $Y:= \Phi_{\sigma}(\langle M \rangle)$ (defined in \eqref{Phi}).
Then 
\[
\bI_{Y}=\langle M \rangle \quad \mbox{where} \quad \bI_{Y}=\{N\in \bK^{c} : \Phi_{\sigma}(\langle N\rangle)\subseteq Y \}.
\]  
\end{theorem}

\begin{proof} Let $\bI=\bI_{Y}$ and $ \bI'=\langle M \rangle$. By definition $\bI'=\langle M \rangle$ is the smallest thick (two-sided) tensor ideal of $\bK^{c}$ containing $M$, so 
it follows that $\bI \supseteq \bI'$.

Now let $N\in\bK$, and apply the exact triangle of functors $\Gamma_{\bI'}\to\text{Id} \to L_{\bI'}\to\ $ to $\Gamma_{\bI}(N)$:
\begin{equation} \label{E:ExactTriangle}
\Gamma_{\bI'}\Gamma_{\bI}(N) \to \Gamma_{\bI}(N) \to L_{\bI'} \Gamma_{\bI}(N) \to
\end{equation}
One can conclude that $\sigma(L_{\bI'}\Gamma_{\bI}(N)) \subseteq Y$ because (i) the first term belongs to $\Loc(\bI') \subseteq \Loc(\bI)$ (since $\bI' \subseteq \bI$) and 
(ii) $\Gamma_{\bI}(N)$ belongs to the triangulated subcategory $\Loc(\bI)$. 

According to Theorem~\ref{T:localizationtriangles}(a) there are no non-zero maps from $\bI'$ to $L_{\bI'}\Gamma_{\bI}(N)$. 
Consequently, for any $S, Q \in\bK^{c}$, one can use (\ref{eq:duality}) to show that   
\begin{equation} \label{E:ExactNTriangle}
0=\bK(S\otimes M \otimes Q, L_{\bI'}\Gamma_{\bI}(N)) \cong \bK(S,  L_{\bI'}\Gamma_{\bI}(N) \otimes Q^* \otimes M^*).
\end{equation}
Since $\bK$ is compactly generated it follows that 
$L_{\bI'}\Gamma_{\bI}(N) \otimes Q^* \otimes M^*=0$ in $\bK$. Hence $L_{\bI'}\Gamma_{\bI}(N) \otimes \bK^c \otimes M^*=0$, and since one can find a set of compact objects $\mc{C}$ with $\Loc(\mc{C})= \bK$, this implies $L_{\bI'}\Gamma_{\bI}(N) \otimes \bK \otimes M^*=0$. One can now conclude the following: 
\begin{align*}
\varnothing &= \Phi_{\sigma}(\llangle L_{\bI'}\Gamma_{\bI}(N)  \otimes \bK \otimes M^*\rrangle) \\
  &= \Phi_{\sigma}(\llangle L_{\bI'}\Gamma_{\bI}(N) \rrangle \otimes \llangle M^*\rrangle)\\
  &=  \Phi_{\sigma}(\llangle L_{\bI'}\Gamma_{\bI}(N) \rrangle) \cap \Phi_{\sigma}(\llangle M \rrangle )  \\
  &\supseteq  \Phi_{\sigma}(\llangle L_{\bI'}\Gamma_{\bI}(N) \rrangle) \cap \Phi_{\sigma}(\langle M \rangle )  \\
  &=  \Phi_{\sigma}(\llangle L_{\bI'}\Gamma_{\bI}(N) \rrangle) \cap Y \\
  &= \Phi_{\sigma}(\llangle L_{\bI'}\Gamma_{\bI}(N) \rrangle).
\end{align*}
The second equality is an application of \leref{MTN}. The third equality uses condition (v) in \deref{nc-w-support}. Therefore, by (\ref{EE:supportone}), 
$L_{\bI'}\Gamma_{\bI}(N) =0$ in $\bK$, and it follows that $\Gamma_{\bI}(N) \cong \Gamma_{\bI'}\Gamma_{\bI}(N)$ via (\ref{E:ExactTriangle}). 

Finally, consider $N\in\bI$. Then by Theorem~\ref{T:localizationtriangles}(b), $\Gamma_{\bI}(N) \cong N$, thus $\Gamma_{\bI'}(N) \cong N$ 
and $N\in\operatorname{Loc}\left( \bI'\right)$. Now by \cite[Lemma 2.2]{Neeman3} we see that in fact $N \in \bI'$. Consequently, $\bI\subseteq\bI'$.
\end{proof}

\section{Classifying thick (two-sided) ideals and Balmer's spectrum of an M$\Delta$C} \label{SS:classifying} 
In this section we present a method for the classification of the thick (two-sided) ideals of an M$\Delta$C
and our first approach towards the explicit description of the Balmer spectrum of an M$\Delta$C as a topological space.
They are based on the use of a weak support datum having the faithfulness and realization properties \eqref{EE:supportone}--\eqref{EE:supporttwo}.
\subsection{The map $\Theta$} Let $\bK$ be a compactly generated M$\Delta$C, and  $\sigma: \bK^{c} \to \XX$ be a weak support datum.  Denote by $\Theta_\sigma$
the map from specialization-closed subsets of $X$ to subsets of $\bK^c$ given by  
\begin{equation}
\label{Theta}
\Theta_\sigma(W)=\{M \in \bK^c : \Phi_{\sigma}(\langle M \rangle) \subseteq W\}.
\end{equation}
The following result verifies that $\Theta_\sigma(W)$ is a thick tensor ideal. 

\bpr{theta-w-ideal} Let $\sigma: \bK^{c} \to \XX$ be a weak support datum for a compactly generated M$\Delta$C, $\bK$.
For any $W \in \mc{X}_{sp}$, $\Theta_\sigma(W)$ is a thick tensor ideal of $\bK^c$.
\epr

\begin{proof}
Since $\langle M \rangle = \langle \Sigma M \rangle$, we have $M \in \Theta_\sigma(W)$ if and only if $\Sigma M \in \Theta_\sigma(W)$. Suppose $M \oplus N \in \Theta_\sigma(W)$. 
Then since $M$ and $N$ are in $\langle M \oplus N \rangle,$ it follows that $M$ and $N$ are in $\Theta_\sigma(W)$. If $A \to B \to C \to \Sigma A$ is a distinguished triangle with 
$B$ and $C$ in $\Theta_\sigma(W)$, then by \leref{dist-triang-weak}, $\Phi_{\sigma}(\langle A \rangle) \subseteq \Phi_{\sigma}(\langle B \rangle) \cup \Phi_{\sigma}(\langle C \rangle)$ 
and so $A \in \Theta_\sigma(W)$. If $M$ is in $\Theta_\sigma(W)$, 
then since $N \otimes M$ and $M \otimes N$ are both in $\langle M \rangle$, we have $N$ and $M$ in $\Theta_\sigma(W)$.
\end{proof}

Suppose $S$ is any subset of the topological space $X$. Denote by 
\begin{equation}
\label{Ssp}
{\mbox{\em{$S_{sp}$ the largest specialization-closed set contained in $S$}}}.
\end{equation}
That is, $S_{sp}$ is the union of all closed sets contained in $S$. With this definition, we can describe the image of $f_\sigma$. 

\bpr{f-theta}
Suppose $\bK$ is a compactly generated M$\Delta$C with a weak support datum $\sigma: \bK^{c} \to \XX$ such that $\Phi_\sigma (\langle C \rangle)$ is closed for every compact object $C$. Then the map $f_\sigma: X \to \Spc \bK^c$ defined in \thref{f-weak} associated to the restriction of $\sigma$ to $\bK^c$
satisfies $f_\sigma(x)= \Theta_\sigma( ( X \backslash \{x\})_{sp}),$ $\forall x \in X$.
\epr

\begin{proof}
Observe that 
\begin{align*}
f_\sigma(x)&= \{ M \in \bK^c : x \not \in \Phi_{\sigma}(\langle M \rangle) \}\\
&= \{ M \in \bK^c: \Phi_{\sigma}(\langle M \rangle) \subseteq X \backslash \{x\}\}\\
&= \{ M \in \bK^c: \Phi_{\sigma}(\langle M \rangle) \subseteq (X \backslash \{x\})_{sp}\}= \Theta_\sigma( (X \backslash \{x\})_{sp}).
\end{align*}
\end{proof}
We end this subsection by recording a useful fact that will be used later. 

\ble{x-minus-x-unique}
Suppose $X$ is a Zariski space. Then, for all $x, y \in X$, 
\[
(X \backslash \{x\})_{sp} = (X \backslash \{y\})_{sp} \quad \Leftrightarrow \quad x=y.
\]
\ele

\begin{proof}
Suppose $(X \backslash \{x\})_{sp} = (X \backslash\{y\})_{sp}$. Then there is no closed set which contains $y$ and not $x$, and vice versa. 
Therefore, $\overline{\{x\}}= \overline{\{y\}}$. In a Zariski space, every irreducible set has a unique generic point, but since $x$ and $y$ 
are generic points of their closures, we have $x=y$ by the assumed uniqueness. 
\end{proof}

\subsection{Classification of thick tensor ideals and the Balmer spectrum}

If $\bK$ is a compactly generated M$\Delta$C with a weak support datum $\sigma$, we have now exhibited maps 
$$
\ThickId(\bK^{c}) \begin{array}{c} {\Phi_{\sigma}} \atop {\longrightarrow} \\ {\longleftarrow}\atop{\Theta_\sigma} \end{array}  \XX_{sp}.
$$
If $X$ is a Zariski space and $\sigma$ satisfies the additional conditions (\ref{EE:supportone}) and (\ref{EE:supporttwo}), 
one can now classify thick tensor ideals of $\bK^c$ and the Balmer spectrum. 

\begin{theorem} \label{I:bijectiongeneral} Let $\bK$ be a compactly generated M$\Delta$C and $\sigma : \bK \to \XX$ be an extended weak support datum
for a Zariski space $X$ such that $\Phi_\sigma (\langle C \rangle)$ is closed for every compact object $C$.
Recall the maps $\Phi_\sigma$ and $\Theta_\sigma$ defined in \eqref{Phi} and \eqref{Theta}, 
and the map $f_\sigma$ from \thref{f-weak} and \prref{f-theta}.
\begin{enumerate}
\item[(a)] If $\sigma$ satisfies the faithfulness property (\ref{EE:supportone}), then $\Theta_\sigma \circ \Phi_\sigma= \id$. 
\item[(b)] If $\sigma$ satisfies the realization property (\ref{EE:supporttwo}), then:
	\begin{enumerate}
	\item[(i)] $\Phi_\sigma \circ \Theta_\sigma= \id$.
	\item[(ii)] The map $f_\sigma$ is injective.
	\end{enumerate}
\item[(c)] If $\sigma$ satisfies both conditions (\ref{EE:supportone}) and (\ref{EE:supporttwo}), then:
	\begin{enumerate}
	\item[(i)]  $\Phi_\sigma$ and $\Theta_\sigma$ are mutually inverse maps. They are isomorphisms of ordered monoids, where the set of thick ideals of $\bK^c$ is equipped 
	 with the operation $\bI, \bJ \mapsto \langle \bI \otimes \bJ \rangle$ and the inclusion partial order, and $\mc{X}_{sp}$ is equipped with the operation of intersection and 
	 the inclusion partial order.
	\item[(ii)] For every prime ideal $\bP$ of $\bK^c$, there exists $x \in X$ with $\Phi_\sigma(\bP)=(X \backslash \{x\})_{sp}$.
	\item[(iii)] The map $f_\sigma: X \to \Spc \bK^c$ is a homeomorphism.
	\end{enumerate}
\end{enumerate}
\end{theorem}

\begin{proof} 

We first show (a). Given a thick tensor ideal $\bI$ of $\bK^c$, set $W= \Phi_\sigma(\bI)$ and $\bI_W= \Theta_\sigma(W)$. Then by definition 
$$I_W= \Theta_\sigma(W)=\Theta_\sigma ( \Phi_\sigma( \bI))= \{ M : \Phi_\sigma(\langle M \rangle) \subseteq \Phi_\sigma(\bI)\} \supseteq \bI.$$

For the reverse inclusion, let $N\in\bI_{W}$, so $\Phi_\sigma(\langle N\rangle) \subseteq W$. Since $X$ is a Zariski space, $\Phi_\sigma(\langle N\rangle)=W_{1}\cup\dots\cup W_{n}$, where 
the $W_{i}$ are the irreducible components of $\Phi_\sigma(\langle N\rangle)$.  Moreover, each $W_{i}$ has a generic point $x_{i}$ with $\overline{\{x_{i}\}}=W_{i}$. 
Since $W_{i}\subseteq W$ one has $x_{i}\in W$. By definition of $W$, there  exists $M_{i}\in\bI$ such that $x_{i}\in \Phi_\sigma(\langle M_{i}\rangle)$. 
Since each $\Phi_\sigma(\langle M_{i}\rangle)$ is closed, it follows that $W_{i}\subseteq \Phi_\sigma(\langle M_{i}\rangle)$. Now set $M:= \bigoplus_{i=1}^{n}M_{i}\in\bI$. Then
$$
\Phi_\sigma(\langle N\rangle) \subseteq \bigcup_{i=1}^{n}\Phi_\sigma(\langle M_{i}\rangle) = \Phi_\sigma(\langle M\rangle) \subseteq W.
$$

We claim that $\langle N\rangle \subseteq \langle M\rangle$. Observe that $\bI$ is a thick tensor ideal containing $\langle M\rangle$, so $\langle M\rangle\subseteq\bI$. 
This implies that the aforementioned assertion will complete the proof of the  inclusion $\bI_{W}\subseteq\bI$.

To prove the claim, we employ Hopkins' Theorem (Theorem~\ref{T:Hopkinstheorem}). By this result one has $\langle M\rangle = \bI_{\Phi_\sigma(\langle M\rangle)}$.
However, $\Phi_\sigma(\langle N\rangle)\subseteq \Phi_\sigma(\langle M\rangle)$, so $\langle N\rangle \subseteq \bI_{\Phi_\sigma(\langle M\rangle)}=\langle M\rangle$.

Next, we show (b)(i). We have automatically that
$$
\Phi_\sigma(\Theta_\sigma(W)) = \Phi_\sigma(\bI_{W}) = \bigcup_{M\in\bI_{W}} \Phi_\sigma(\langle M\rangle) \subseteq W.
$$
For the reverse inclusion, express $W=\bigcup_{j\in J} W_{j}$ for some index set $J$ and closed subsets $W_{j}\in \XX$. By the assumption 
(\ref{EE:supporttwo}), there exist objects $N_{j}\in \bK^{c}$ such that $\Phi_\sigma(\langle N_{j}\rangle) = W_{j}$ for 
$j\in J$. It follows that $N_{j}\in \bI_{W}$ so $W \subseteq \bigcup_{M\in\bI_{W}} \Phi_\sigma(\langle M\rangle)$. Consequently, $\Phi_\sigma(\Theta_\sigma(W))=W$.

For (b)(ii), we just note that by (b)(i), $\Theta_\sigma$ is injective. By \leref{x-minus-x-unique}, the map sending $x \mapsto (X \backslash \{x\})_{sp}$ is injective. By \prref{f-theta}, 
$f_\sigma(x)= \Theta_\sigma( (X\backslash \{x\})_{sp}).$ Hence, $f$ is injective.

By (a) and (b), (c)(i) is automatic. We now show (c)(ii). 
Suppose $\bP$ is a prime ideal. By (b)(i), we can write arbitrary specialization-closed sets as $\Phi_\sigma(\bI)$ and $\Phi_\sigma(\bJ)$ for some ideals $\bI$ and $\bJ$. We have
$$\Phi_\sigma(\bI) \cap \Phi_\sigma(\bJ) \subseteq \Phi_\sigma(\bP)$$
$$\Updownarrow$$
$$\Phi_\sigma(  \bI \otimes \bJ  )\subseteq \Phi_\sigma(\bP)$$
$$\Updownarrow$$
$$ \bI \otimes \bJ \subseteq \bP $$
$$\Updownarrow$$
$$\bI \text{ or }\bJ \subseteq \bP$$
$$\Updownarrow$$
$$\Phi_\sigma(\bI) \text{ or } \Phi_\sigma(\bJ) \subseteq \Phi_\sigma(\bP).$$
Hence, $\Phi_\sigma(\bP)$ has the property that for any specialization-closed sets $S$ and $T$ of $X$, 
$S \cap T \subseteq \Phi_\sigma(\bP) \Rightarrow  S$ or $T \subseteq \Phi_\sigma(\bP)$. We claim that the only sets with this property are sets of the form 
$(X \backslash \{x\})_{sp}$. Suppose $\Phi_\sigma(\bP)$ is not a set of this form. Then for every point $x$ in its complement, there exists some closed 
set $V_x$ which does not contain $x$ and is not contained in $\Phi_\sigma(\bP)$. We have
$\bigcap_{x \in \Phi_\sigma(\bP)^c} V_x \subseteq \Phi_\sigma(\bP),$ but for each $x$, $V_x \not \subseteq \Phi_\sigma(\bP)$. 
By assumption $X$ is Noetherian, so there is a finite subset $S$ of $\Phi_\sigma(\bP)^c$ with $\bigcap_{x \in \Phi_\sigma(\bP)^c} V_x = \bigcap_{x \in S} V_x$, 
and since this is now a finite intersection this shows that there exist closed sets $S$ and $T$ with $S \cap T \subseteq \Phi_\sigma(\bP)$, 
but neither $S$ nor $T$ is contained in $\Phi_\sigma(\bP)$. Since this is a contradiction, $\Phi_\sigma(\bP)$ has the form $(X \backslash \{x\})_{sp}$ for some $x$. 

Now we will show (c)(iii), that $f_\sigma$ is a homeomorphism. Given a prime ideal $\bP$ of $\bK$, there exists $x \in X$ such that $\Phi_\sigma(\bP)=(X \backslash \{x\})_{sp}$, and so
$$\bP = \Theta_\sigma( \Phi_\sigma ( \bP)) = \Theta_\sigma( (X \backslash \{x\})_{sp})= f_\sigma(x)$$ by \prref{f-theta}. This shows that $f_\sigma$ 
is surjective, and hence bijective by (b)(ii). We now show that $f_\sigma$ is a closed map. To an arbitrary closed set, which by (\ref{EE:supporttwo}) 
is of the form $\Phi_\sigma(\langle M \rangle)$, we apply $f_\sigma$:
$$f_\sigma( \Phi_\sigma(\langle M \rangle)) = f_\sigma ( f^{-1}_\sigma(V(M))) = V(M),$$
using the surjectivity of $f_\sigma$ and the formula for $f_\sigma^{-1}(V(M))$ given in \thref{f-weak}. Since $V(M)$ is closed, 
$f_\sigma$ is a closed and continuous bijection, and hence a homeomorphism.
\end{proof}

\section{Noncommutative geometry for one-sided ideals} \label{S:onesided} 

In this section we present a method for the classification of the thick right ideals of an M$\Delta$C.
We introduce a new concept (quasi support datum) to deal with thick one-sided ideals.

In the next section we illustrate how one can extract an explicit description of the 
Balmer spectrum of an M$\Delta$C and a classification of its thick two-sided ideals out of the 
classification of thick one-sided ideals given in this section.
%As in the case with two-sided ideals there are additional assumptions needed beyond the basic definition of quasi-support. 
%In the next section we demonstrate how one can pass from thick right ideals to thick two-sided ideals of an M$\Delta$C, and use the 
%results of this section as a second method to describe the (noncommutative) Balmer spectrum of an M$\Delta$C.
\subsection{Quasi support data} 
A thick right ideal of an M$\Delta$C, $\bK$, is a full triangulated subcategory of $\bK$ that contains all direct summands of its objects and is closed under right
tensoring with arbitrary objects of $\bK$.
\bde{nc-q-support}
Let $\bK$ be a monoidal triangulated category, $X$ a topological space, and $\sigma$ a map $\bK \to \mc{X}$. We call $\sigma$ a {\em (noncommutative) quasi support datum} if
\begin{enumerate}
\item[(a)] $\sigma(0)=\varnothing$ and $ \sigma(1)= X$; 
\item[(b)] $\sigma(A\oplus B)=\sigma(A)\cup \sigma(B)$ $\forall A, B \in \bK$; 
\item[(c)] $\sigma(\Sigma A)=\sigma(A)$, $\forall A \in \bK$; 
\item[(d)] If $A \to B \to C \to \Sigma A$ is a distinguished triangle, then $\sigma(A) \subseteq \sigma(B) \cup \sigma(C)$; 
\item[(e)] $\sigma(A \otimes B)\subseteq \sigma (A)$, $\forall A, B \in \bK$. 
\end{enumerate}
\ede
\noindent 
For an extended quasi support datum we replace (b) with 
\begin{enumerate} 
\item[(b')] $\sigma(\bigoplus_{i \in I} A_i)=\bigcup_{i \in I} A_i$, $\forall A_i \in \bK$.
\end{enumerate} 

Similarly to the previous two sections, we will be interested in quasi support data $\sigma:\bK \to \mc{X}$ that satisfy the following one-sided type assumptions: 
\begin{eqnlist} 
%\item \label{E:supportone} $\sigma(M)\in {\mathcal X}_{cl}$ for $M\in \bK^{c}$; 
\item \label{E:supporttwo} $\Phi(\llangle M \rrangle_{\mathrm{r}})=\varnothing$ if and only if $M=0$, for all $M \in \bK$ (Faithfulness Property); 
\item \label{E:supportthree} For any $W\in {\mathcal X}_{cl}$ there exists  $M\in \bK^{c}$ such that $\Phi(\langle M\rangle_{\mathrm{r}})=W$ (Realization Property). 
\end{eqnlist}

Here and below, similarly to the two-sided case, for $M \in \bK^c$, $\langle M \rangle_{\mathrm{r}}$ denotes the smallest thick right ideal of $\bK^c$ containing $M$; that is the intersection of 
all thick right ideals containing $M$. For $M \in \bK$, $\llangle M \rrangle_{\mathrm{r}}$ denotes the smallest thick right ideal of $\bK$ containing $M$. 

\subsection{One-sided Hopkins' Theorem for quasi support data} We first state an assumption that acts as a (one-sided) replacement for 
%the consequence presented in Corollary~\ref{phi-ideal} 
condition (e) in the definition of weak support datum. Recall the definition \eqref{Phi} of the map $\Phi_\sigma$. 
Similarly to the arguments in Section \ref{SS:weaksupport}, one shows that 
in the presence of the other conditions for quasi support datum, condition (e) is equivalent to 
\begin{equation}
\label{q-sup-sig-Phi}
\Phi_\sigma ( \llangle M \rrangle_{\mathrm{r}} ) = \sigma (M), \quad \forall M \in \bK.  
\end{equation}

\begin{assumption} \label{A:keyproj} Suppose that $M, N \in \bK^{c}$ are such that 
\[
\Phi_\sigma( \langle N \rangle_{\mathrm{r}}) \subseteq \sigma(M).
\]
Set $\bI'=\langle M \rangle_{\mathrm{r}}$. If $M^* \otimes L_{\bI'}(N) =0$, then $L_{\bI'} (N)=0$ (for the localization functor as in Section \ref{loc-fun}).
\end{assumption} 

With this assumption, one proves the following one-sided version of Theorem \ref{T:Hopkinstheorem}.
%can replace the part of the argument that employs Corollary~\ref{phi-ideal} with Assumption~\ref{A:keyproj}.
Similarly to \eqref{Theta}, for a quasi support datum $\sigma:\bK \to \mc{X}$ for a compactly generated M$\Delta$C $\bK$, 
denote by $\Theta_\sigma$ the map from specialization-closed subsets of $X$ to subsets of $\bK^c$:
\[
\Theta_\sigma(W)=\{M \in \bK^c : \Phi_{\sigma}(\langle M \rangle_{\mathrm{r}}) \subseteq W\}, \quad \mbox{for} \quad W \in \mc{X}_{sp}.
\]

\begin{theorem} \label{T:Hopkinstheorem2} Let $\bK$ be a compactly generated M$\Delta$C and $\sigma:\bK \to \mc{X}$ be an assignment to subsets of a Zariski space $X$ that satisfies the conditions 
(a), (b'), (c), (d) for an extended quasi support datum, such that $\sigma:\bK^{c} \to \mc{X}$ is a quasi support datum for a Zariski space $X$. Assume that Assumption~\ref{A:keyproj} holds. Then for each object $M\in\bK^{c}$,     
\[
\Theta_\sigma ( \Phi_\sigma ( \langle M \rangle_{\mathrm{r}})) = \langle M \rangle_{\mathrm{r}}. 
\]
\end{theorem}
\subsection{Classification of thick tensor one-sided (right) ideals} With Theorem~\ref{T:Hopkinstheorem2}, we can state a classification 
theorem for thick (right) ideals for $\bK^{c}$. The proof follows the same line of reasoning as given in Theorem~\ref{I:bijectiongeneral}. 

\begin{theorem} \label{I:bijectiongeneral-r}  Let $\bK$ be a compactly generated M$\Delta$C and $\sigma:\bK \to \mc{X}$ be an assignment to subsets of a Zariski space $X$ that satisfies the conditions 
(a), (b'), (c), (d) for an extended quasi support datum. Suppose that $\sigma$ restricts to a quasi support datum on $\bK^{c}$ where $\Phi_\sigma (\langle C \rangle)$ is closed for every $C\in \bK^{c}$. 
Moreover, assume that $\sigma$ satisfies the realization property (7.1.\ref{E:supportthree}) and Assumption~\ref{A:keyproj} holds. 

Then the maps 
$\Phi_\sigma$ and $\Theta_\sigma$
$$
\{\text{thick right ideals of $\bK^{c}$}\} \begin{array}{c} {\Phi_\sigma} \atop {\longrightarrow} \\ {\longleftarrow}\atop{\Theta_\sigma} \end{array}  \XX_{sp}
$$
are mutually inverse. 
\end{theorem}
\bre{q-support-no-monoid} The set of thick right ideals of $\bK^{c}$ is an ordered monoid with the operation $\bI, \bJ \mapsto \langle \bI \otimes \bJ \rangle_{\mathrm{r}}$ and the inclusion partial order. The set $\mc{X}_{sp}$ is  is an ordered monoid with the operation of intersection and the inclusion partial order. The maps $\Phi_\sigma$ and $\Theta_\sigma$ preserve inclusions but in general they are not isomorphisms of monoids. 

More precisely, {\em{$\Phi_\sigma$ and $\Theta_\sigma$ are isomorphisms of ordered monoids if and only if $\sigma : \bK \to \mc{X}$  is a support datum.}} 

Indeed, $\Phi_\sigma$ is an isomorphism of monoids if and only if $\Phi_\sigma( \langle \bI \otimes \bJ \rangle_{\mathrm{r}}) = \Phi_\sigma( \bI) \cap  \Phi_\sigma(\bJ)$
for all thick right ideals $\bI$ and $\bJ$ of $\bK^{c}$. This in turn is equivalent to $\Phi_\sigma( \bI \otimes \bJ ) = \Phi_\sigma( \bI) \cap  \Phi_\sigma(\bJ)$
by an argument similar to the proof of \leref{gamma-ideal}. Since $\bI = \cup_{A \in \bI}  \langle A \rangle_{\mathrm{r}}$, the last property 
is equivalent to 
$\Phi_\sigma(  \langle A \rangle_{\mathrm{r}} \otimes  \langle B \rangle_{\mathrm{r}} ) = \Phi_\sigma( \langle A \rangle_{\mathrm{r}}) \cap  \Phi_\sigma( \langle B \rangle_{\mathrm{r}})$,
$\forall A, B \in \bK$. By \eqref{q-sup-sig-Phi} the last property is equivalent to 
\[
\bigcup_{C \in \bK} \sigma (A   \otimes C \otimes B) = \sigma (A) \cap \sigma(B), \quad \forall A, B \in \bK,
\]
which is the fifth property in the definition of support data.
\qed
\ere

\subsection{Finite-dimensional Hopf algebras}\label{S:fdHopfalgrightideals}

Let $A$ be a non-semisimple finite-dimensional Hopf algebra over a field $\kk$. It is well known that $A$ is a self-injective algebra. 
Denote by $\Mod(A)$ the abelian category of (possibly infinite-dimensional) $A$-modules and by $\modd(A)$ its abelian subcategory of finite-dimensional $A$-modules.
Both categories are monoidal with an exact tensor product. Both categories are Frobenius. 
Denote by $\bK= \StMod(A)$ and $\stmod(A)$ the corresponding stable categories. $\bK= \StMod(A)$ is a compactly generated M$\Delta$C with 
$\bK^c = \stmod(A)$, see \cite[Example 1.48]{BIK2}. 

We first state a condition that provides a method for constructing a quasi-support data for $\bK^{c}=\stmod(A)$. 

\begin{assumption} \label{A:fg} Given $A$ a finite-dimensional Hopf algebra. 
\begin{itemize} 
\item[(a)] The cohomology ring $R=\operatorname{H}^{\bullet}(A,\kk)$ is a finitely generated algebra. 
\item[(b)] Given $M, N \in \modd (A)$, $\operatorname{Ext}_{A}^{\bullet}(M,N)$ is a finitely generated $R$-module. 
\end{itemize} 
\end{assumption} 

Assumption~\ref{A:fg} will be referred to as the {\em (fg) assumption}. This assumption is conjectured to hold in broad generality by Etingof and Ostrik, that is, for all finite tensor categories, which includes categories of modules of finite-dimensional Hopf algebras \cite[Conjecture 2.18]{EO1}. Given the (fg) assumption, one can construct the 
{\em cohomological support datum} on $\bK^{c}$ as follows. Let $M$ be an object of $\bK^{c}$ and 
let 
\begin{equation}
\label{coh-supp}
\sigma(M)=\text{Proj}(\{P\in \operatorname{Spec}(R):\ \text{Ext}_{A}(M,M)_{P}\neq 0\}).
\end{equation}
Clearly, $\sigma(M)\in {\mathcal X}_{cl}$ for $M\in \bK^{c}$. 
Next we will verify (7.1.\ref{E:supporttwo}). 
One has $\Phi_\sigma(\langle M \rangle_{\mathrm{r}})=\varnothing$ if and only if $\sigma(N)=\varnothing$ for all $N\in \langle M \rangle_{\mathrm{r}}$ if and only if 
$N=0$ for all $N\in \langle M \rangle_{\mathrm{r}}$ (cf. \cite[Proposition 2.3]{FW1}) if and only if $M=0$. 

Next, we can verify (7.1.\ref{E:supportthree}). Let $\zeta\in \operatorname{H}^{n}(A,{\mathbb C})$ and let $L_{\zeta}$ be the kernel of the map 
$\zeta:\Omega^{n}({\mathbb C})\rightarrow {\mathbb C}$. Moreover, let $Z(\zeta)=\{P\in \operatorname{Proj}(\operatorname{Spec}(R)):\ 
\zeta\in P\}$. One has $\sigma(L_{\zeta})=Z({\zeta})$. We have the following result from \cite[Theorem 2.5]{FW1}. 

\begin{proposition}\label{T:Lzetatensor} Let $M\in \stmod(A)$, and $\zeta, \zeta_{i}\in R$ be homogeneous elements of positive degree ($1\le i\le t)$. 
\begin{itemize} 
\item[(a)] $\sigma(L_{\zeta}\otimes M)=\sigma(L_{\zeta})\cap \sigma(M)$.
\item[(b)]  $\sigma(\otimes_{i=1}^{t} L_{\zeta_{t}}\otimes M)=[\cap_{i=1}^{t} \sigma(L_{\zeta})]\cap \sigma(M)$. 
\end{itemize}
\end{proposition}

Now let $W\in {\mathcal X}_{cl}$. There exists $\zeta_{1},\zeta_{2},\dots,\zeta_{t}\in R$ such that $W=Z(\zeta_{1})\cap Z(\zeta_{2})\cap  \dots \cap Z(\zeta_{t})$. 
Then by Proposition~\ref{T:Lzetatensor}, it follows that $W=\sigma(L_{\zeta_{1}}\otimes L_{\zeta_{2}}\otimes \dots \otimes L_{\zeta_{t}})$ (cf. \cite[Remark 2.7]{FW1}). 

In \cite{BIK1}, Benson, Iyengar and Krause constructed an extension of the cohomological support map $\sigma$ to the set of objects of $\bK= \StMod(A)$. This map 
takes values in the subsets of the space of homogeneous prime ideals of $\operatorname{H}^{\bullet}(A,\kk)$ and satisfies 
conditions (a), (b'), (c), (d) for an extended quasi support datum. The map
\[
A \in \bK \mapsto \sigma(A) \cap \operatorname{Proj}(\operatorname{H}^{\bullet}(A,\kk))
\]
is an extension of $\sigma$ to $\bK$ with values in $\XX(\operatorname{Proj}(\operatorname{H}^{\bullet}(A,\kk)))$, which still satisfies conditions (a), (b'), (c), (d).

With the verifications above, we can now state the following result. 

\begin{theorem} \label{T:fdHopfalgebras} Let $A$ be a finite-dimensional Hopf algebra over a field $\kk$ that satisfies (fg) and Assumption~\ref{A:keyproj}. 
With the cohomological support datum there exists a bijection 
between 
$$
\{\text{thick right ideals of $\stmod(A)$}\} \begin{array}{c} {\Phi_\sigma} \atop {\longrightarrow} \\ {\longleftarrow}\atop{\Theta_\sigma} \end{array}  
\XX_{sp}(\operatorname{Proj}(\operatorname{H}^{\bullet}(A,\kk))).
$$
\end{theorem}

\section{Small quantum group for Borel subalgebras} 
\label{smallq-Borel}
In this section, we use the method from Section~\ref{S:onesided} to classify the thick right ideals of the stable module categories of the small quantum groups for the Borel subalgebras of 
all complex simple Lie algebras. Based on these results, 
we then provide a classification of the thick two-sided ideals of these categories and an explicit description of their Balmer's spectra.

\subsection{Small quantum group for Borel subalgebras} Let ${\mathfrak g}$ be a complex simple Lie algebra and $\zeta$ be a primitive $\ell$th root of unity in ${\mathbb C}$. Assume that 
$\ell >h$ where $h$ is the Coxeter number associated to the root system of ${\mathfrak g}$. Let ${\mathfrak b}$ be the Borel subalgebra of ${\mathfrak g}$ 
corresponding to taking the negative root vectors and ${\mathfrak u}$ be the unipotent radical of ${\mathfrak b}$. Let $u_{\zeta}({\mathfrak b})$ 
(resp. $u_{\zeta}({\mathfrak g})$) be the small quantum group for ${\mathfrak b}$ (resp. ${\mathfrak g}$). Moreover, let $U_{\zeta}({\mathfrak b})$ (resp. 
$U_{\zeta}({\mathfrak g})$) be the Lusztig ${\mathcal A}$-forms specialized at ${\zeta}$. The later algebras are often referred to as the big quantum group. 

The Steinberg module $\text{St}$ is a projective and injective module in $\modd(U_{\zeta}({\mathfrak g}))$. This fact allows one to 
prove that there exists enough projectives in the category and the module category is self injective. Therefore, one can consider the 
stable module category of finite-dimensional modules, $\stmod(U_{\zeta}({\mathfrak g}))$. This category is a symmetric M$\Delta$C. 
A classification of thick tensor ideals and the computation of the Balmer spectrum was achieved in \cite{BKN2}. 

The Steinberg module $\text{St}$ does not remain projective upon restriction to $\text{mod}(U_{\zeta}({\mathfrak b}))$. Thus, the 
construction of the stable module category does not follow immediately from the aforementioned argument. The small quantum 
group is a finite-dimensional Hopf algebra which means that it is self-injective, and in this case one can discuss the stable module category. 

For $\stmod(u_{\zeta}({\mathfrak g}))$ the classification of thick tensor ideals and computation of the Balmer spectrum remain as 
open questions. In this section we will address these issues for $\stmod(u_{\zeta}({\mathfrak b}))$. 

\subsection{Action of $\Pi$} \label{S:actionpi} Denote for brevity
\[
A=u_{\zeta}({\mathfrak b}).
\]
Let $X_1$ denote $X(T)/ \ell X(T)$, where $X(T)$ is the weight lattice corresponding to a maximal split torus $T$. For each $\lambda \in X_{1}$, define an automorphism $\gamma_\lambda$ of $A$ which is defined by $\gamma_\lambda (E_\alpha)= \zeta^{\langle \lambda, \alpha \rangle} E_\alpha$ for all $\alpha \in \Phi^{+}$ and $\gamma_\lambda(K_i)=K_i$ for $i=1, \dotsc ,n$. Here $\Phi^{+}$ are the set of positive roots in the set of roots $\Phi$ and 
$\langle -, - \rangle$ denotes the inner product on the Euclidean space spanned by the roots.

Let $\Pi= \{ \gamma_{\lambda} :\ \lambda \in X_{1} \}$, which is a group under composition. Moreover, $\Pi$ acts on the cohomology ring $R=\operatorname{H}^{2\bullet}(A,\kk)\cong S^{\bullet}({\mathfrak u}^{*})$. The action of $\lambda$ is given explicitly on homogeneous elements of $R$ (interpreted as $n$-fold extensions) by sending each module $M$ to the twist $M^{\gamma_\lambda}$, a new module with action $a.m=\gamma_\lambda(a) m$ for all $a \in A$ and $m \in M^{\gamma_\lambda}.$ Since $\kk^{\gamma_\lambda} \cong \kk$, which follows from the definitions of the coproduct and counit on $A$, this action sends an $n$-fold extension of $\kk$ by $\kk$ to another such extension.

The following result demonstrates that the action of $\Pi$ on the stable module category for the small quantum Borel algebra is trivial. Let 
$X(T)/{\mathbb Z}\Phi$ denote the weight lattice modulo the root lattice, and $({\ell},|X(T)/{\mathbb Z}\Phi |)=\text{gcd}({\ell},|X(T)/{\mathbb Z}\Phi |)$. 

\begin{theorem}  
\label{T:triviality-Pi-act}
Let $u_{\zeta}({\mathfrak b})$ be the small quantum group for the Borel subalgebra ${\mathfrak b}$ with ${\ell}>h$ and $({\ell},|X(T)/{\mathbb Z}\Phi |)=1$.  
Then the action of $\Pi$ on $\operatorname{Proj}(S^{\bullet}({\mathfrak u}^{*}))$ is trivial. 
\end{theorem} 

\begin{proof} It suffices to prove that the action of $\Pi$ on the cohomology 
ring $R$ is trivial. This action can be described as follows. 

If $\lambda \in X_{1}$, $a\in {\mathbb C}^{*}$, one can define an automorphism $\widehat{\gamma}_{\lambda}$ of $A$: $\widehat{\gamma}_\lambda (E_\alpha)= a^{\langle \lambda, \alpha \rangle} E_\alpha$ for all $\alpha \in \Phi^{+}$ and $\widehat{\gamma}_\lambda(K_i)=K_i$ for $i=1, \dotsc ,n$. Note that this is also an automorphism on the subalgebra, $u_{\zeta}({\mathfrak u})$ which is generated by $E_{\alpha}$, $\alpha\in \Phi^{+}$. When $a=\zeta$ where $\zeta$ is a primitive ${\ell}$th root of unity, $\widehat{\gamma}=\gamma$. 

This automorphism $\widehat{\gamma}$ acts on the bar resolution of $u_{\zeta}({\mathfrak u})$ (cf. \cite[Section 2.8]{BNPP}) and the differentials respect the action of this automorphism. Therefore, the automorphism will act on the cohomology $\operatorname{H}^{\bullet}(u_{\zeta}({\mathfrak u}),{\mathbb C})$. 
One has that $R=\operatorname{H}^{\bullet}(u_{\zeta}({\mathfrak u}),{\mathbb C})^{u_{\zeta}({\mathfrak t})}=S^{\bullet}({\mathfrak u}^{*})^{(1)}$ \cite[Theorem 2.5]{GK1} where $u_{\zeta}({\mathfrak t})$ is the subalgebra of $A$ generated by $K_{i}$ for $i=1, \dotsc, n$. The twist means that all the weights in the cohomology of $R$ are of the form ${\ell} \sigma$ where $\sigma$ is in the weight lattice. From the bar resolution, we can conclude that ${\ell}\sigma$ must be in the root lattice. Since $({\ell},|X(T)/{\mathbb Z}\Phi |)=1$, it follows that 
$\sigma$ is in the root lattice.  

Let $f\in R$ of weight ${\ell}\sigma$. Then $f$ can be expressed as a sum of tensor products of elements in $u_{\zeta}({\mathfrak u})$ such that ${\ell}\sigma=\sum_{\alpha\in \Phi^{+}} n_{\alpha}\alpha$. Therefore, the automorphism $\widehat{\gamma}_{\lambda}(f)=a^{\langle \lambda, {\ell}{\sigma} \rangle}=(a^{{\ell}})^{\langle \lambda, \sigma \rangle}$. 
Specializing $a=\zeta$ shows that $\gamma$ acts trivially on $R$. 
\end{proof} 

\subsection{Classification of one-sided (right) ideals:} The small quantum group $A=u_{\zeta}({\mathfrak b})$ is a finite-dimensional Hopf algebra over ${\mathbb C}$. For ${\ell}>h$ ($h$ is the Coxeter number for the underlying root system), the (fg) assumption holds and the cohomology 
ring $\operatorname{H}^{2\bullet}(A,{\mathbb C})\cong S^{\bullet}({\mathfrak u}^{*})$. The odd degree cohomology is zero. 

The Assumption~\ref{A:keyproj} holds by \cite[Section 7.4]{BKN2} as long as $W(M)=W(M^{*})$ for every $M\in \stmod(u_{\zeta}({\mathfrak b}))$, where $W=\sigma$ is 
the cohomological support. We can prove this under the assumption that $\Pi$ acts trivially on $R$ (i.e., $({\ell},|X(T)/{\mathbb Z}\Phi |)=1$). In checking 
Assumption~\ref{A:keyproj} for $M=0$, we also use the well-known faithfulness property for the cohomological support. 

First note that $M\in \stmod(u_{\zeta,\Gamma}({\mathfrak b}))$ is rigid, thus $M$ is a summand of $M \otimes M^* \otimes M$, and 
\[
W(M) \subseteq W(M \otimes M^* \otimes M).
\]
Using the fact that  $M$ has a composition series by subquotients isomorphic to the one dimensional modules 
$\lambda\in X_{1}$, one has 
\[
W(M \otimes M^* \otimes M) = \bigcup_{\lambda\in X_{1}} W(\lambda \otimes M^* \otimes M).
\]
The cohomological support $W$ is a quasi support datum, and $\Pi$ act trivially on supports. Consequently, 
\[
W(\lambda \otimes M^* \otimes M) \subseteq W((M^*)^{\gamma_\lambda} \otimes \lambda \otimes M) \subseteq W((M^*)^{\gamma_\lambda}) = \gamma_\lambda^{-1}\cdot W(M^*)=W(M^*)
\]
for all $\lambda\in X_{1}$. Combining the above inclusions gives $W(M) \subseteq W(M^*)$. The other inclusion is proved by interchanging $M$ with $M^{*}$ and using 
the fact that $(M^{*})^{*} \cong M$. 

We can therefore employ Theorem~\ref{T:fdHopfalgebras} to give a classification of one-sided tensor ideals for the stable module category 
of $u_{\zeta}({\mathfrak b})$. 

\begin{theorem} \label{T:one-sidedborel} Let $u_{\zeta}({\mathfrak b})$ be the small quantum group for the Borel subalgebra ${\mathfrak b}$ with ${\ell}>h$ and 
$\gcd({\ell},|X(T)/{\mathbb Z}\Phi |)=1$. Then there exists a bijection between 
$$
\{\text{thick right ideals of $\stmod(u_{\zeta}({\mathfrak b}))$}\} \begin{array}{c} {\Phi} \atop {\longrightarrow} \\ {\longleftarrow}\atop{\Theta} \end{array}  
\{\text{specialization closed sets of $\operatorname{Proj}(S^{\bullet}({\mathfrak u}^{*}))$} \}. 
$$
\end{theorem}

\subsection{Classification of two-sided ideals and the Balmer spectrum} 
First, we fix some notation:
\begin{enumerate}
\item[(i)] The cohomology ring of $A=u_{\zeta}(\mathfrak{b})$ will be denoted by
$$R=\Ext^\bullet_A(\kk,\kk)=\operatorname{H}^\bullet(A, \kk).$$
\item[(ii)] The space of (nontrivial) homogeneous prime ideals of $R$ will be denoted by
$$X=\Proj R.$$
The set of subsets, closed subsets, and specialization-closed subsets of $X$ will be denoted respectively by $\mc{X}, \mc{X}_{cl},$ and $\mc{X}_{sp}$.
\item[(iii)] We will use the support function on $\StMod (A)$ from \cite{BIK1}, using the cohomology ring $R$ defined above. It is defined using the localization and colocalization functors from Theorem \ref{T:localizationtriangles}.
This support will be denoted by
$$\widetilde W (-): \StMod(A) \to \mc{X}_{sp}.$$ 
It extends the corresponding cohomological support functions from \eqref{coh-supp}. 

\item[(iv)] The map $\Phi$ associated to the support $\widetilde W$ will be denoted
by $\Phi_{\widetilde{W}}$. It takes thick subcategories of $\StMod(A)$ to subsets of 
$\mc{X}$. 
\end{enumerate}

In order to explicitly describe the Balmer spectrum of $\stmod(A)$, we must produce a 
support having the faithfulness and realization properties \eqref{EE:supportone}--\eqref{EE:supporttwo}. 
By \cite[Theorem 5.2]{BIK1}, the weak support datum $\widetilde{W} : \StMod(A) \to \mc{X} (\Proj (R))$
has the faithfulness property \eqref{EE:supportone}.
However, to get the realization property (\ref{EE:supporttwo}), we will need to consider a new support datum built from $\widetilde W$. 

Each $\lambda \in X_1$ corresponds to a one-dimensional simple module, which we will denote again by $\lambda$, where $E_\alpha$ acts by 0, and $K_i$ acts by $\zeta^{\langle \lambda, \alpha_i \rangle}.$ All the simple modules of $A$ arise in this way. Using the definition of the coproduct on $A$, one can directly verify that there exists an isomorphism of $A$-modules:
\begin{equation} \label{e:twistiso} 
\lambda \otimes Q \otimes \lambda^{-1} \cong Q^{\gamma_{\lambda}}.
\end{equation} 
Hence, the action of $\Pi$ on $R$ can be realized as conjugation by the module $\lambda$ under the tensor product. 

From the action of $\Pi$ on $R$, one can construct the stack quotient $X_{\Pi}:=\Pi$-$\operatorname{Proj}(R)$,
the space of nonzero homogeneous $\Pi$-prime ideals of $R$, as studied by Lorenz in \cite{Lorenz1}. These are nonzero $\Pi$-invariant homogeneous ideals $P$ of $R$ that have the property 
$I J \subseteq P \Rightarrow I \subseteq P$ or $J \subseteq P$ for 
all $\Pi$-invariant homogeneous ideals $I,J$ of $R$. The stack quotient $X_\Pi$ is a Zariski space by the argument in \cite[Section 2.3]{BKN1}.
The space of $\Pi$-orbits in $\Proj(R)$ will be denoted by
$$\widetilde X_\Pi=\Pi\backslash \Proj(R).$$ By \cite[Section 1.3]{Lorenz1}, there are maps
\begin{center}
\begin{tikzcd}
                                      &  &                                               & \mf{p} \arrow[lllddd, maps to] \arrow[rrrddd, maps to]          &     &  &                                 \\
                                      &  &                                               & X \arrow[ld, "\pi_1", two heads] \arrow[rd, "\pi", two heads] &     &  &                                 \\
                                      &  & \widetilde X_{\Pi} \arrow[rr, "\pi_2", two heads] &                                                                 & X_\Pi &  &                                 \\
\Pi\cdot \mf{p} \arrow[rrrrrr, maps to] &  &                                               &                                                                 &     &  & \bigcap_{x \in \Pi} x\cdot \mf{p}
\end{tikzcd}
\end{center}
and the topologies on $\widetilde X_\Pi$ and $X_\Pi$ are defined to be the final topologies with respect to the surjections from $X$.

Denote 
$$
W=\pi \circ \widetilde W : \StMod(A) \to \mc{X}(X_\Pi).
$$ 
This is a quasi support datum.
Then $\Phi_W$ will denote the associated map given by \eqref{Phi}.

The quasi support datum $W$ clearly satisfies $W(M)\in {\mathcal X}_{cl}(X_\Pi)$ for $M\in \bK^{c}$.
Next we will verify (5.1.\ref{E:supporttwo}). 
One has $\Phi_W(\langle M \rangle)=\varnothing$ if and only if $W(N)=\varnothing$ for all $N\in \langle M \rangle$ if and only if 
$N=0$ for all $N\in \langle M \rangle$  if and only if $M=0$. Note that the last if and only if statement 
holds because if $P$ is a projective $A$-module then $C\otimes P$ and $P\otimes C$ are projective $A$-modules for any $C\in \bK^{c}$.

Finally, let $M$ be an object in $\bK^{c}$. Note that by the quasi support property on tensor products and the fact that $\lambda \otimes \lambda^{-1} \cong \kk$, we have
$$\widetilde W(\lambda \otimes M) = \widetilde W( \lambda \otimes M \otimes \lambda^{-1} \otimes \lambda ) \subseteq \widetilde W( \lambda \otimes M \otimes \lambda^{-1})  \subseteq \widetilde W( \lambda \otimes M),$$
which then implies by (\ref{e:twistiso}) that
\begin{equation} \label{e:qtens-eq}
\widetilde W(\lambda \otimes M) = \widetilde W( M^{\gamma_\lambda}).
\end{equation} 
Now we can observe
$$\Phi_{\widetilde{W}}(\langle M \rangle)=\bigcup_{C,D\in \bK^{c}} \widetilde{W}(C\otimes M \otimes D)=\bigcup_{C\in \bK^{c}} \widetilde{W}(C\otimes M).$$ 
The last equality follows because $\widetilde W$ is a quasi support datum. Next observe that since the simple modules for $A$ are given by $\{\lambda: \lambda\in X_{1}\},$ by the property of the quasi support on 
short exact sequence, and by (\ref{e:twistiso}) and (\ref{e:qtens-eq}), one obtains 
$$\bigcup_{C\in \bK^{c}} \widetilde{W}(C\otimes M)=\bigcup_{\lambda\in X_{1}} \widetilde{W}(\lambda \otimes M)=\bigcup_{\lambda\in X_{1}} \widetilde{W}(M^{\gamma_{\lambda}})=
\bigcup_{\lambda\in X_{1}} \gamma_{\lambda}^{-1}\cdot \widetilde{W}(M).$$ 
Therefore, 
\begin{equation} 
\Phi_{\widetilde{W}}(\langle M \rangle)=\Pi\cdot \widetilde{W}(M).
\end{equation}  
We can conclude that $\Phi_W$ takes two-sided ideals in $\bK^{c}$ to specialization closed sets in $X_\Pi=\Pi$-$\operatorname{Proj}(R)$, and that the extension $W$ of $\widetilde{W}$ satisfies the conditions \eqref{EE:supportone}--\eqref{EE:supporttwo}. 

At this point we will assume that the action of $\Pi$ on the cohomology ring $R$ is trivial. Then (i) $W=\widetilde{W}$ and (ii) $\widetilde{W}(T)=\widetilde{W}(T^{*})$ for any $T\in \stmod(u_{\zeta}({\mathfrak b}))$. 
Since we know only that $\widetilde{W}$ is a quasi support, rather than a weak support, one needs to verify an appropriate version of Assumption~\ref{A:keyproj} for Hopkins' Theorem to hold. 
The statement needed to apply the proof in Theorem~\ref{T:Hopkinstheorem} is given as follows: if  $L_{\bI'}\Gamma_{\bI}(N) \otimes Q^* \otimes M^*=0$ for all $Q\in \stmod(u_{\zeta}({\mathfrak b}))$ then 
$L_{\bI'}\Gamma_{\bI}(N)=0$. In the proof of Theorem~\ref{T:Hopkinstheorem}, it was shown that 
$$\widetilde{W}(L_{\bI'}\Gamma_{\bI}(N))\subseteq Y=\Phi_{\widetilde{W}}(\langle M \rangle)=\bigcup_{Q}\widetilde{W}(Q\otimes M).$$ 
Since $L_{\bI'}\Gamma_{\bI}(N) \otimes Q^* \otimes M^*=0$, it follows by \cite[Theorem 6.2.1]{BKN2} that 
$$\varnothing=\widetilde{W}(L_{\bI'}\Gamma_{\bI}(N))\cap [\cup_{Q} \widetilde{W}(Q^* \otimes M^*)]=\widetilde{W}(L_{\bI'}\Gamma_{\bI}(N))\cap [\cup_{Q} \widetilde{W}(Q \otimes M)]=
\widetilde{W}(L_{\bI'}\Gamma_{\bI}(N))\cap Y.$$ 
Therefore, $\widetilde{W}(L_{\bI'}\Gamma_{\bI}(N))=\varnothing$ and $L_{\bI'}\Gamma_{\bI}(N)=0$. 

Now one can apply the argument given in Theorem~\ref{I:bijectiongeneral}(c)(i) to conclude that $\Phi_{\widetilde{W}}$ is an order-preserving bijection between thick two-sided ideals of $\stmod(u_{\zeta}(\mf{b}))$ and specialization-closed sets of $\Proj(S^\bullet(\mf{u}^*))$. Then this implies that $\widetilde{W}$ is not just a quasi support, but in fact a weak support, by the following argument, which is a noncommutative version of \cite[Proposition 6.2.1]{BKN2}. 

First observe that every thick ideal $\bI$ of $\stmod(u_{\zeta}(\mf{b}))$ is semiprime. As noted in the proof of Lemma \ref{dual-ideal}, each compact object $V$ is a direct summand of $V \otimes V^* \otimes V$. This implies 
\[
V \otimes \stmod(u_{\zeta}(\mf{b})) \otimes V \subseteq \bI \quad \Rightarrow  \quad V \otimes V^* \otimes V \in \bI \quad \Rightarrow \quad V \in \bI,
\] and by \thref{semiprime-equiv}, $\bI$ is semiprime. 

Next by the semiprimeness of $\langle \bI \otimes \bJ \rangle$ and $\bI \cap \bJ$, we have inclusions 
\begin{center}
\begin{tikzcd}
\langle \bI \otimes \bJ \rangle \arrow[r, hook] \arrow[d, no head, Rightarrow, no head] & \bI \cap \bJ \arrow[d, Rightarrow, no head]              \\
\bigcap_{\bI \otimes \bJ \subseteq \bP}\bP                                     & \bigcap_{\bI \cap \bJ \subseteq \bP} \bP \arrow[l, hook']
\end{tikzcd}
\end{center}
where the intersections range over $\bP$ in $\Spc(\stmod(u_{\zeta}(\mf{b})))$ satisfying the given conditions; hence $\langle \bI \otimes \bJ \rangle = \bI \cap \bJ$. Therefore,
\begin{align*} 
\Phi_{\widetilde{W}}(\langle \bI \otimes \bJ \rangle)&= \Phi_{\widetilde{W}}(\bI \cap \bJ)\\
& = \Phi_{\widetilde{W}}(\bI) \cap \Phi_{\widetilde{W}}(\bJ),
\end{align*}
where the second equality follows from the fact that $\Phi_{\widetilde{W}}$ is an order-preserving bijection between thick ideals and specialization-closed sets of $\Proj(S^\bullet(\mf{u}^*))$, and so $\widetilde{W}$ is a weak support datum. The theorem below now follows immediately.

\begin{theorem} \label{T:two-sidedborel-Balmer} Let $u_{\zeta}({\mathfrak b})$ be the small quantum group for the Borel subalgebra ${\mathfrak b}$ with ${\ell}>h$ and 
$\gcd({\ell},|X(T)/{\mathbb Z}\Phi |)=1$. The thick two-sided ideals of $\stmod(u_{\zeta}({\mathfrak b}))$ coincide with the thick one-sided ideals of $\stmod(u_{\zeta}({\mathfrak b}))$.
\begin{itemize} 
\item[(a)] Then there exists a bijection between 
$$
\{\text{thick two-sided ideals of $\stmod(u_{\zeta}({\mathfrak b}))$}\} \begin{array}{c} {\Phi} \atop {\longrightarrow} \\ {\longleftarrow}\atop{\Theta} \end{array}  
\{\text{specialization closed sets of $\operatorname{Proj}(S^{\bullet}({\mathfrak u}^{*}))$} \}. 
$$
\item[(b)] There exists a homeomorphism $f:$ $\operatorname{Proj}(S^{\bullet}({\mathfrak u}^{*})) \to \Spc(\stmod(u_{\zeta}({\mathfrak b})))$. 
\end{itemize} 
\end{theorem}

We expect that there exists a wide class of finite-dimensional Hopf algebras $A$ satisfying the (fg) assumption
for which 
\begin{enumerate}
\item[(a)] the one-sided thick ideals of $\stmod(A)$ are in bijection with the specialization closed subsets of $\operatorname{Proj}(\operatorname{H}^{\bullet}(A,\kk))$ 
as in Theorem \ref{T:fdHopfalgebras}, while  
\item[(b)] the spectrum $\Spc (\stmod(A))$ is homeomorphic to $\Pi$-$\operatorname{Proj}(\operatorname{H}^{\bullet}(A,\kk))$ 
and the two-sided thick ideals of $\stmod(A)$ are in bijection with the specialization closed subsets of $\Pi$-$\operatorname{Proj}(\operatorname{H}^{\bullet}(A,\kk))$, where
$\Pi$ is the group of invertible objects of  $\stmod(A)$.  
\end{enumerate}
The small quantum Borel $u_{\zeta}({\mathfrak b})$ is very special in that the action of $\Pi$ on $\operatorname{Proj}(\operatorname{H}^{\bullet}(u_{\zeta}({\mathfrak b}),\kk))$
is trivial. This is not the case in general as illustrated in the next section with the Benson--Witherspoon Hopf algebras. 

\section{Benson--Witherspoon Hopf algebras} 
\label{BW}

In this section, we use the method of Section \ref{SS:classifying}, to give an explicit description of the Balmer spectra of the stable 
module categories of the Benson--Witherspoon Hopf algebras \cite{BW1} and a classification of their thick two-sided ideals.

\subsection{The smash coproduct of a group algebra and group coordinate ring} The Benson--Witherspoon Hopf algebras are 
the Hopf duals of smash products of a group algebra and a coordinate ring of a group. They were studied in \cite{BW1}.

In more detail, let $G$ and $H$ be finite groups, 
%(maybe assume in addition that $G$ is a $p$-group), 
with $H$ acting on $G$ by group automorphisms. 
Let $\kk$ be a field of positive characteristic dividing the order of $G$.
Define $A$ as the Hopf algebra dual to the smash product $\kk[G]\# \kk H$, 
where $\kk [G]$ is the dual of the group algebra of $G$, and $\kk H$ is the group algebra of $H$. 
Denote by $p_g$ the dual basis element of $\kk [G]$ corresponding to $g \in G$. By definition, this smash product is $\kk [G] \otimes \kk H$ 
as a vector space, and multiplication is given by 
$$ (p_g \otimes x )(p_h \otimes y) = p_g (x_{(1)} . p_h) \otimes x_{(2)}y =p_g p_{x.h} \otimes xy= \delta_{g, x.h} p_g \otimes xy$$ 
for all $g \in G$ and $x,y \in H$. Now define $$A=\Hom_{\kk} (\kk[G] \# \kk H, \kk).$$
As an algebra, $A=\kk G \otimes \kk [H]$. The comultiplication in $A$ is given by 
$$\Delta( g \otimes p_x)= \sum_{y \in H} (g \otimes p_y) \otimes (y^{-1} . g \otimes p_{y^{-1} x}).$$ 
The counit and antipode are given by $$\epsilon (g \otimes p_x) = \delta_{x,1} \quad \mbox{and} \quad S(g \otimes p_x)= x^{-1}. (g^{-1}) \otimes p_{x^{-1}}.$$ 
Note that while $G$ is a subalgebra of $A$, via the map $g \mapsto g \otimes 1$, it is not a Hopf subalgebra, since 
$$\Delta_G (g)= g \otimes g$$ and 
$$\Delta_A( g \otimes 1)= \sum_{x \in H} \Delta_A (g \otimes p_x)= \sum_{x,y \in H} (g \otimes p_y) \otimes (y^{-1} .g \otimes p_{y^{-1} x}) \not = (g \otimes 1) \otimes (g \otimes 1).$$

An $A$-module is the same as an $H$-graded $\kk G$-module. We may write any $A$-module $M$ in the form $$M=\bigoplus_{x \in H} M_x \otimes \kk_x,$$ where the $M_x$ are $\kk G$-modules. The action of $\kk G$ is on the first tensor and $\kk [H]$ acts on the second. 

In \cite{BW1}, Benson and Witherspoon prove the following formula for the decomposition of a tensor product of $A$-modules:
$$(M \otimes \kk_x) \otimes (N \otimes \kk_y)= (M \otimes {}^{x}N) \otimes \kk_{xy}$$ on homogeneous components. 
Here and below for $M \in \Mod(\kk G)$ and $x \in H$, ${}^x M \in \Mod(\kk G)$ denotes the conjugate of $M$ by the action of $x \in H \to \Aut(G)$. 
On homogeneous components, the dual of a module is given by 
$$(M \otimes \kk_x)^*={}^{x} (M^*) \otimes \kk_{x^{-1}}.$$ 

\subsection{Support data for $\StMod (A)$ and $\StMod (\kk G)$}

By the definition of the smash product, we have an embedding of Hopf algebras $$\kk [G] \hookrightarrow \kk [G] \# \kk H.$$ Hence, when we dualize to the smash coproduct, we get a Hopf algebra surjection $$ \kk G \twoheadleftarrow A.$$ We will use the following notation: 
\begin{enumerate}
\item[(i)] The cohomology rings of $A$ and $\kk G$ will be denoted by
$$R_A= \Ext^\bullet_A (\kk, \kk) \quad \mbox{and} \quad R_G= \Ext^\bullet_G (\kk, \kk),$$
respectively.
\item[(ii)] Denote the spaces
$$X^A= \Proj R_A \quad \mbox{and} \quad X^G= \Proj R_G.$$
The collections of their specialization closed subsets and all subsets will be denoted respectively by
$$\mc{X}^A_{sp}, \; \mc{X}^A, \; \mc{X}^G_{sp}, \; \mc{X}^G.$$
\item[(iii)] We will use the support functions on $\StMod (A)$ and $\StMod (\kk G)$ from \cite{BIK1}, where the relevant ring $R$ is taken to be 
$R_A$ and $R_G$, respectively. They are defined using the localization and colocalization functors from Theorem \ref{T:localizationtriangles}.
They take values in the sets of all subsets of the spaces of homogeneous prime ideals of $R_A$ and $R_G$, respectively.
By the discussion after the proof of Lemma 10.1 in \cite{BIK1}, the second extension takes values in $\mc{X}^G$:
$$W_G(-): \StMod(\kk G) \to \mc{X}^G.$$ 
In other words, the irrelevant ideal of $R_G$ is not in the image $W_G(M)$ for any $M \in \StMod (\kk G)$. 
The first extension takes values in $\mc{X}^A$ by \thref{ra-rg-iso}(c)
$$\widetilde W_A (-): \StMod(A) \to \mc{X}^A.$$ 
These support maps extend the corresponding cohomological support functions from \eqref{coh-supp}.

\item[(iv)] The map $\Phi$ associated to the support $\widetilde W_A$ will be denoted
by $\widetilde{\Phi}_A$. It takes thick subcategories of $\StMod(A)$ to subsets of 
$\mc{X}^A$. 
\item[(v)] The functor $\Mod(\kk G) \to \Mod(A)$ defined on objects by $$M \mapsto M \otimes \kk_e$$ will be denoted by $\mc{F}$. 
%Clearly, it is a fully faithful tensor functor.
\end{enumerate}

\bre{no-strong-tensor}
In \cite{BW1}, Benson and Witherspoon give an example of $A$-modules $M$ and $N$ such that $M \otimes N$ is projective and $N \otimes M$ 
is not projective. 
This shows that in general, the (strong) tensor product property does not hold for the cohomological support $\widetilde  W_A$. In other words, it is not necessarily 
true that $\widetilde W_A(M \otimes N)= \widetilde W_A(M) \cap \widetilde W_A(N)$. Furthermore, there are no support data maps in the sense 
of Balmer $\sigma : \StMod (A) \to  \mc{X}$ such that $\sigma(A)=\varnothing \Leftrightarrow A \cong 0$. (If there were such maps, then 
$\sigma (M \otimes N) = \sigma(M) \cap \sigma (N) = \sigma(N \otimes M)$, so $M \otimes N \cong 0 \Leftrightarrow N \otimes M \cong 0$.) 
\ere

\ble{f-monoid-exact}
%Let $A= \Hom_{\kk}(\kk[G] \# \kk H, \kk)$ where $G$ and $H$ are finite groups with $H$ acting on $G$, and $\mc{F}: \Mod(\kk G) \to \Mod(A)$ 
%the functor $M \mapsto M \otimes \kk_e$. Then 
The functor $\mc{F}$ descends to a fully faithful tensor triangulated functor $\For : \StMod(\kk G) \to \StMod(A)$.
\ele

\begin{proof}
By the formula for the tensor product of $A$-modules, $\mc{F}$ is monoidal, since $$\mc{F}(M \otimes N) \cong \mc{F}(M) \otimes \mc{F}(N).$$ It is clear that $\mc{F}$ is exact, and it is fully faithful since morphisms of $A$-modules are the same as graded morphisms of $\kk G$-modules. The functor $\mc{F}$ descends to a functor $\StMod(\kk G) \to \StMod(A)$ because it sends projective modules to projective modules, and has the property that for each morphism $f$ in $\Mod(\kk G)$, if $\mc{F}(f)$ factors through a projective module in $\Mod(A)$, then $f$ factors through a projective module in $\Mod(\kk G)$.
\end{proof}

Denote by $\For : \Mod(A) \to \Mod(\kk G)$ the forgetful functor. 
It is clear that it descends to a tensor triangulated functor $\StMod(\kk G) \to \StMod(A)$.

\bth{ra-rg-iso} For all  Benson--Witherspoon Hopf algebras $A$ the following hold:
%Let $A= \Hom_{\kk}(\kk[G] \# \kk H, \kk)$ where $G$ and $H$ are finite groups with $H$ acting on $G$, 
%and $R_G=\Ext_{ G}^\bullet(\kk, \kk)$ and $R_A=\Ext_A^\bullet(\kk, \kk)$ the cohomology rings of $G$ and $A$, respectively. 
\begin{enumerate}
\item[(a)]There is a canonical isomorphism $R_G \cong R_A$. (Denote $R:=R_G \cong R_A$).
\item[(b)] If $C$ and $Q$ are $\kk G$-modules, there is an isomorphism of $R$-modules $$\Ext_G^\bullet(C,Q) \cong \Ext_A^\bullet (\mc{F}(C), \mc{F}(Q)),$$ 
and $A$ satisfies the (fg) Assumption \ref{A:fg}.
%which will allow us to define the Hopf algebra support for $A$. 
\item[(c)] For an $A$-module $N$,
$$\widetilde W_A(N) = W_G(\For(N)).$$
\item[(d)] For an $A$-module $Q$, 
$$
\widetilde \Phi_A( \langle Q \rangle)= H \cdot W_G(\For(Q)).
$$
\end{enumerate}
\eth

\begin{proof}
For (a) and (b), suppose $$0 \to \mc{F}(Q) \to N_1 \to... \to N_i \to \mc{F}(C) \to 0$$ is an exact sequence representing an element of $\Ext_A^\bullet(\mc{F}(C),\mc{F}(Q))$. Then we claim it is equivalent to an exact sequence which is supported only at the identity component. To do this, we may just note that the natural maps give an equivalence of extensions:
\begin{center}
\begin{tikzcd}
0 \arrow[r] & \mc{F}(Q) \arrow[d] \arrow[r] & N_1 \arrow[d] \arrow[r] & ... \arrow[r] & N_i \arrow[d] \arrow[r] & \mc{F}(C) \arrow[d] \arrow[r] & 0 \\
0 \arrow[r] & \mc{F}(Q) \arrow[r]           & (N_1)_e \arrow[r]       & ... \arrow[r] & (N_i)_e \arrow[r]       & \mc{F}(C) \arrow[r]           & 0
\end{tikzcd}
\end{center}
This gives a vector space isomorphism $$\Ext_A^\bullet(\mc{F}(Q), \mc{F}(C)) \cong \Ext_G^\bullet(Q,C).$$ This isomorphism is compatible with the actions of $R_A$ and $R_G$ because $\mc{F}$ is a monoidal functor. This decomposition allows us to conclude Assumption \ref{A:fg} for $A$, since it is well-known that this assumption holds for $\kk G$.

For (c), write $N= \bigoplus_{z \in H} N_z \otimes \kk_z$ with $N_z \in \modd (\kk G)$. Note that by \cite[Theorem 5.2]{BIK1}, 
$$\widetilde W_A(N)= \bigcup_{C \text{ compact}} \min \Hom_{\StMod(A)}^\bullet (C,N),$$
where, for an $R_A$-module $L$, $\min L$ refers to the minimal primes in the support of $L$. Hence, 
\begin{align*}
\widetilde W_A(N)&= \bigcup_{C \in \modd(\kk G), z \in H} \min \Hom_{\StMod (A)}^\bullet (C \otimes \kk_z, N)\\
&= \bigcup_{C,z} \min \Hom^\bullet_{\StMod(A)} (C \otimes \kk_z, N_z \otimes \kk_z)\\
&= \bigcup_{C,z} \min \Hom^\bullet_{\StMod(A)} (C \otimes \kk_e, (N_z \otimes \kk_z) \otimes (\kk \otimes \kk_z)^*)\\
&= \bigcup_{C,z} \min \Hom^\bullet_{\StMod(A)} (C \otimes \kk_e, N_z \otimes \kk_e)\\
&= \bigcup_{C,z} \min\Hom^\bullet_{\StMod(\kk G)} (C,N_z)
= \bigcup_z W_G (N_z).
\end{align*}
The second to last equality follows from the fact that for $i>0,$ $$\Hom^i_{\StMod(A)} ( C \otimes \kk_e, N_z \otimes \kk_e) \cong \Ext^i_{A} (C \otimes \kk_e , N_z \otimes \kk_e)$$ by \cite[Proposition 2.6.2]{CTVZ1}, which is isomorphic to $\Ext^i_{G} (C , N_z)$ by (2). Additionally, for $i=0$ we have $$\Hom_{\StMod(A)} (C \otimes \kk_e, N_z \otimes \kk_e) \cong \Hom_{\StMod(\kk G)} (C, N_z)$$ since the functor $\mc{F}$ is fully faithful.

For (d), we have
\begin{align*}
\widetilde \Phi_A( \langle Q \rangle)&=\bigcup_{M, N,x,y,z} \widetilde W_A ( (M_x \otimes \kk_x) \otimes (Q_z \otimes \kk_z) \otimes (N_y \otimes \kk_y))\\
&= \bigcup_{M, N, x, y , z} \widetilde W_A (( M_x \otimes ^x Q_z \otimes ^{xz} N_y ) \otimes \kk_{xzy})\\
&= \bigcup_{M, N, x, y, z} W_G (  M_x \otimes ^x Q_z \otimes ^{xz} N_y)\\
&= \bigcup_{M, N, x, y, z} \left ( W_G (  M_x ) \cap W_G ( ^{x } Q_z) \cap  W_G ( ^{xz} N_y) \right )\\
&= \bigcup_{x,z} x. W_G(Q_z)
= H \cdot (W_G (\For(Q))).
\end{align*}

\end{proof}

\bco{int-h-invariance}
For all Benson--Witherspoon Hopf algebras $A$, 
\[
\widetilde W_A : \StMod(A) \to \mc{X}^A
\]
is an extended weak support datum on $\StMod(A)$ satisfying the faithfulness condition \eqref{EE:supportone}.
\eco

\begin{proof} The fact that $\widetilde W_A$ satisfies conditions (a)--(d) in Definition \ref{dnc-w-support} follows from Theorem \ref{tra-rg-iso}(c) 
and the fact that $W_G$ is a support datum for $\StMod(\kk G)$.  For condition (e) in Definition \ref{dnc-w-support}, 
we need to verify the property $$\widetilde \Phi_A ( \langle M \rangle \otimes \langle N \rangle ) = \widetilde \Phi_A( \langle M \rangle) \cap \widetilde \Phi_A ( \langle N \rangle).$$
This follows as both sides are equal to 
$$
[ H \cdot W_G( \For (M) ) ] \cap [ H \cdot W_G( \For (N) ) ] 
$$
by Theorem \ref{tra-rg-iso}(d). 

To check the faithfulness of $\widetilde W_A$, assume that $M = \oplus_{x \in H} M_x \otimes \kk_x \in \Mod(A)$ is such that $\widetilde \Phi_A( \langle M \rangle)= \varnothing$. 
Applying Theorem \ref{tra-rg-iso}(d), gives that $H \cdot W_G(\For(M)) = \varnothing$. By the  faithfulness of $\widetilde W_A$, $\For(M) = \oplus_{x \in H} M_x$ 
is a projective $\kk G$-module, and thus, $M_x$ are projective $\kk G$-modules for all $x \in H$. This implies that $M$ is a projective $A$-module.
\end{proof}

%\bco{Phi-sup}
%For all Benson--Witherspoon Hopf algebras $A$, $\widetilde W_A : \StMod(A) \to \mc{X} (\Proj (R_A))$
%is a quasi support datum on $\StMod(A)$ that satsifies
%\[
%\widetilde \Phi_A ( \langle M\rangle_{\mathrm{r}}  \otimes  \langle N \rangle_{\mathrm{r}} ) = 
%\widetilde \Phi_A( \langle M \rangle_{\mathrm{r}}) \cap \big( H \cdot \widetilde \Phi_A ( \langle N \rangle_{\mathrm{r}}) \big).
%\]
%\eco
%\begin{proof} The first part of the statement follows from \thref{ra-rg-iso}(c). The second part is proved analogously to \thref{ra-rg-iso}(d).
%\end{proof}
\subsection{Classification of thick two-sided ideals and the Balmer spectrum of $\stmod(A)$}

In order to explicitly describe the Balmer spectrum of $\stmod(A)$, we must produce a weak 
support datum having the faithfulness and realization properties \eqref{EE:supportone}--\eqref{EE:supporttwo}. 
By \cite[Theorem 5.2]{BIK1}, the weak support datum $\widetilde{W}_A : \StMod(A) \to \mc{X}^A= \mc{X} (\Proj (R_A))$
has the faithfulness property \eqref{EE:supportone}.
However, to get the realization property (\ref{EE:supporttwo}), we will need to consider a new support datum built from $\widetilde W_A$. 
Denote
$$X_H=H\text{-}\Proj(R_A),$$
the space of nonzero homogeneous $H$-prime ideals of $A$ in the sense of Lorenz \cite{Lorenz1},
i.e. nonzero $H$-invariant homogeneous ideals $P$ of $R_A$ that have the property $I J \subseteq P \Rightarrow I \subseteq P$ or $J \subseteq P$ for 
all $H$-invariant homogeneous ideals $I,J$ of $R_A$. 
% $H$-invariant homogeneous prime ideals of $A$,
$X_H$ is a Zariski space by the argument in \cite[Section 2.3]{BKN1}.
The topological space of $H$-orbits in $X^A = \Proj(R_A)$ will be denoted by
$$\widetilde X_H=H\backslash \Proj(R_A).$$ By \cite[Section 1.3]{Lorenz1}, there are maps
\begin{center}
\begin{tikzcd}
                                      &  &                                               & \mf{p} \arrow[lllddd, maps to] \arrow[rrrddd, maps to]          &     &  &                                 \\
                                      &  &                                               & X^A \arrow[ld, "\pi_1", two heads] \arrow[rd, "\pi", two heads] &     &  &                                 \\
                                      &  & \widetilde X_H \arrow[rr, "\pi_2", two heads] &                                                                 & X_H &  &                                 \\
H\cdot \mf{p} \arrow[rrrrrr, maps to] &  &                                               &                                                                 &     &  & \bigcap_{h \in H} h\cdot \mf{p}
\end{tikzcd}
\end{center}
and the topologies on $\widetilde X_H$ and $X_H$ are defined to be the final topologies with respect to the surjections from $X_A$.

Denote 
$$
W_A =\pi \circ \widetilde W_A : \StMod(A) \to \mc{X}(X_H).
$$ 
Denote by $\Phi_A$ the associated map $\Phi_{W_A}$ map given by \eqref{Phi}.

\ble{wa-extra-cond} For all Benson--Witherspoon Hopf algebras $A$, $W_A$ is a weak support datum satisfying the faithfulness and realization 
conditions \eqref{EE:supportone}--\eqref{EE:supporttwo}.
\ele

\begin{proof} Since $\widetilde{W}_A$ is a weak support datum satisfying the faithfulness property, the same is true for 
$W_A=\pi \circ \widetilde{W}_A$.

Because $X_H$ is equipped with the final topology with respect to $\pi$, and the preimage of $W_A(Q)=\pi (\widetilde W_A(Q))$ is $\widetilde W_A(Q)$, 
which is closed, we have that $W_A(Q)$ is closed in $X_H$. 

Let $Y \subseteq X_H$ be closed. Then $\pi^{-1}(Y)$ is a closed $H$-stable subset of $X_A$. This implies there is some $Q$ with $$W_G(Q)=\pi^{-1}(Y).$$
Since $\pi^{-1}(Y)$ is $H$-stable, using \thref{ra-rg-iso}, we may check
\begin{align*}
\Phi_A( \langle \mc{F}(Q) \rangle)&= \pi \circ \widetilde \Phi_A( \langle \mc{F} (Q) \rangle )= \pi ( H \cdot W_G(Q))\\
&= \pi (H \cdot \pi^{-1}(Y)) = \pi( \pi^{-1}(Y)) =Y.
\end{align*}
Hence, $W_A(-)$ also satisfies the realizability property.
\end{proof}

Applying Theorem \ref{I:bijectiongeneral} we obtain:

\bth{smash-coprod-balmer-spc}
Let $A= \Hom_{\kk}(\kk[G] \# \kk H, \kk)$ where $G$ and $H$ are finite groups with $H$ acting on $G$ and $\kk$ is a base field of positive characteristic dividing
the order of $G$. Let $R_A$ be the cohomology ring of $A$, i.e. $R_A=\Ext^\bullet_A(\kk, \kk)$. The following hold:
\begin{enumerate} 
\item[(a)] There exists a bijection 
$$
\{\text{thick two-sided ideals of $\stmod(A)$}\} \begin{array}{c} {\Phi_A} \atop {\longrightarrow} \\ {\longleftarrow}\atop{\Theta_A} \end{array}  
\{\text{specialization closed sets of $H$-$\Proj(R_A)$} \},
$$
where $\Theta_A$ is the map given by \eqref{Theta} for the weak support datum $W_A$. 
\item[(b)] There exists a homeomorphism $f: H$-$\Proj(R_A) \to \Spc(\stmod(A))$. 
\end{enumerate} 
\eth

\end{document}